\documentclass[a4paper, 11pt]{amsart}
\usepackage{mathpazo} 
\usepackage[utf8]{inputenc}
\usepackage[T1]{fontenc}

\usepackage[dvipsnames]{xcolor}
\usepackage[backref=page,colorlinks,bookmarksopen=true,allcolors=MidnightBlue,linktocpage=true]{hyperref}
\usepackage{amsfonts,amsmath,amssymb,amsthm,stmaryrd}
\usepackage{nicefrac}
\usepackage{enumerate}
\usepackage{pgf,tikz}
\usetikzlibrary{cd}
\usetikzlibrary{arrows}

\usepackage{mathrsfs, xspace}
\usepackage{a4wide}
\usepackage[all]{xy}

\theoremstyle{plain}
\newtheorem{proposition}{Proposition}[section]
\newtheorem{theorem}[proposition]{Theorem}

\newtheorem{lemma}[proposition]{Lemma}
\newtheorem{corollary}[proposition]{Corollary}

\newtheorem{cor}[proposition]{Corollary}

\newtheorem{prop}[proposition]{Proposition}
\newtheorem{lem}[proposition]{Lemma}
\newtheorem{thm}[proposition]{Theorem}

\theoremstyle{definition}
\newtheorem{definition}[proposition]{Definition}
\newtheorem{defn}[proposition]{Definition}

\theoremstyle{remark}
\newtheorem{remark}[proposition]{Remark}
\newtheorem{example}[proposition]{Example}
\newtheorem{examples}[proposition]{Examples}
\newtheorem{question}[proposition]{Question}

\everymath{\displaystyle}

\DeclareMathOperator{\Stab}{Stab}

\DeclareMathOperator{\trace}{Trace}

\DeclareMathOperator{\Span}{Span}

\DeclareMathOperator{\Id}{Id}
\DeclareMathOperator{\Ad}{Ad}

\newcommand{\Exterior}{\mathchoice{{\textstyle\bigwedge}}%
    {{\bigwedge}}%
    {{\textstyle\wedge}}%
    {{\scriptstyle\wedge}}}

\DeclareMathOperator{\SO}{SO}
\DeclareMathOperator{\Sp}{Sp}
\DeclareMathOperator{\Isom}{Isom}

\newcommand{\cat}{{\upshape CAT(0)}\xspace}

\newcommand{\action}{\curvearrowright}

\newcommand{\calV}{\mathcal V}
\newcommand{\bor}{\mathbf{or}}

\newcommand{\R}{\mathbf{R}}                          
\newcommand{\Q}{\mathbf{Q}}                          
\newcommand{\C}{\mathbf{C}}                          
\newcommand{\N}{\mathbf{N}}                          
\newcommand{\Prob}{\mathbf{P}}

\DeclareMathOperator{\LL}{L}
\DeclareMathOperator{\PP}{P}
\DeclareMathOperator{\SSS}{S}
\DeclareMathOperator{\OOO}{O}

\DeclareMathOperator{\OO}{O_\mathbf K}
\DeclareMathOperator{\OC}{O_\mathbf C}
\DeclareMathOperator{\OH}{O_\mathbf H}
\DeclareMathOperator{\OR}{O_\mathbf R}

\DeclareMathOperator{\GL}{GL}
\DeclareMathOperator{\PU}{PU}
\DeclareMathOperator{\SU}{SU}
\DeclareMathOperator{\U}{U}
\DeclareMathOperator{\PO}{PO}
\renewcommand{\H}{\mathbf H}
\newcommand{\K}{\mathbf K}

\newcommand{\calX}{\mathcal X}
\newcommand{\ov}{\overline}
\newcommand{\calY}{\mathcal Y}
\newcommand{\calT}{\mathcal T}
\newcommand{\G}{\Gamma}
\newcommand{\cal}{\mathcal}

\newcommand{\Ss}{\mathcal S}
\newcommand{\calW}{\mathcal W}
\newcommand{\Xx}{\mathcal X}

\newcommand{\Hh}{\mathcal H}
\newcommand{\Yy}{\mathcal Y}
\newcommand{\calC}{\mathcal C}
\newcommand{\calF}{\mathcal F}
\newcommand{\calH}{\mathcal H}
\newcommand{\calZ}{\mathcal Z}
\newcommand{\calI}{\mathcal I}

\newcommand{\rk}{{\rm rk}} 

\newcommand{\HH}{{\rm H}} 
\newcommand{\Hb}{{\rm H}_b} 
\newcommand{\Cb}{{\rm C}_b} 
\newcommand{\Hcb}{{\rm H}_{cb}} 
\newcommand{\Hc}{{\rm H}_{c}} 
\newcommand{\HCn}{\Xx_\C(1,n)} 
\newcommand{\HCk}{\Xx_\C(1,k)} 
\newcommand{\Ii}{\mathcal I}

\renewcommand{\P}{\mathbf P}

\renewcommand{\SS}{\mathbf S} 
\DeclareMathOperator{\PSL}{PSL}
\DeclareMathOperator{\Jac}{Jac}

\author{Bruno Duchesne}
\address{Université de Lorraine, CNRS, Institut \'Elie Cartan de Lorraine, F-54000 Nancy, France}
\author{Jean L\'ecureux}
\address{Laboratoire de Mathématiques d’Orsay, Univ. Paris-Sud, CNRS, Université Paris-Saclay, 91405 Orsay, France}
\author{Maria Beatrice Pozzetti}
\address{Mathematisches Institut, Universität Heidelberg, 69120 Heidelberg, Deutschland}
\thanks{The two first authors are supported in part by French projects ANR-14-CE25-0004 GAMME and ANR-16-CE40-0022-01 AGIRA, the last author is supported in part by the DFG priority program SPP 2026 Geometry at infinity}
\date{\today}
\begin{document}
\title[Boundary maps and maximal representations in infinite dimension]{Boundary maps and maximal representations on infinite dimensional Hermitian symmetric spaces}
\begin{abstract}
We define a Toledo number for actions of surface groups and complex hyperbolic lattices on infinite dimensional Hermitian symmetric spaces, which allows us to define maximal representations.
 When the target is not of tube type we show that there cannot be Zariski-dense maximal representations, and whenever the existence of a boundary map can be guaranteed, the representation preserves a finite dimensional totally geodesic subspace on which the action is maximal. In the opposite direction we construct examples of geometrically dense maximal representation in the infinite dimensional Hermitian symmetric space of tube type and finite rank. Our approach is based on the study of  boundary maps, that we are able to construct in low ranks or under some suitable Zariski-density assumption, circumventing the lack  of local compactness in the infinite dimensional setting.
\end{abstract}

\maketitle
\section{Introduction}
\subsection{Representations of semi-simple Lie groups and their lattices.}
Lattices (i.e. discrete subgroups with finite covolume) of semi-simple Lie groups may be thought of as discretizations of these Lie groups. The question of knowing how much of the ambient group is encoded in its lattices is very natural and attracted a lot of interest in the past decades. 

Among many results, one can spotlight Mostow’s strong rigidity that implies that a lattice   in a higher rank semi-simple Lie group without compact factors completely determines the Lie group \cite{MR0385004}. Later Margulis proved his superrigidity theorem and showed that linear representations of irreducible lattices of higher rank semi-simple algebraic groups over local fields are ruled by representations of the ambient algebraic groups \cite{MR1090825}.

These rigidity results may be understood using a geometric object associated to the algebraic group: a Riemannian symmetric space (for a Lie group) or a Euclidean building (for an algebraic group over a  non-Archimedean field).

Lattices have natural and interesting linear representations outside the finite dimensional world, starting with Hilbert spaces. For example, some representations may come from the principal series of the Lie group. Outside the world of unitary representations, some infinite dimensional representations of a lattice have a very strong geometric flavor. This is the case when there is an invariant non-degenerate quadratic or Hermitian form of finite index, that is when the representation falls in $\PO_\K(p,\infty)$ where $\K=\R$ or $\C$ and $p$ is finite. Then, one can consider the associated action on some infinite dimensional Riemannian symmetric space of non-positive curvature $\calX_\K(p,\infty)$. For example, when $p=1$, $\calX_\K(p,\infty)$ is the infinite dimensional real or complex hyperbolic space. Gromov had these expressive words about $\calX_\R(p,\infty)$ \cite[p.121]{MR1253544}:\\

\begin{quote}
\textit{These spaces look as cute and sexy to me as their finite dimensional siblings but they have been neglected by geometers and algebraists alike.}\\
\end{quote}

In \cite{MR3343349}, an analogue of Margulis superrigidity has been obtained for higher rank cocompact lattices of semi-simple Lie groups using harmonic maps techniques.  The main result is that non-elementary representations preserve a totally geodesic copy of a finite dimensional symmetric space of non-compact type. The finite rank assumption, here $p<\infty$, may be thought of as a geometric Ersatz of local compactness.

The reader should be warned that even in the case of actions on finite rank symmetric spaces of infinite dimension, some new baffling phenomena may appear. For example, Delzant and Py exhibited representations of  $\PSL_2(\R)$ in $\OOO_\R(1,\infty)$ (and more generally of $\PO(1,n)$ in $\OOO_\R(p,\infty)$ for some  values of $p$ depending on $n$). They found a one parameter family of exotic deformations of  $\calX_\R(1,2)$ in $\calX_\R(1,\infty)$ equivariant with respect to representations leaving no finite dimensional totally geodesic subspace invariant. See \cite{MR2881312}, and \cite{MR3263898} for a classification. Very recently, this classification has been extended to self-representations of $\OOO_\R(1,\infty)$ \cite{MP18}. Moreover, exotic representations of $\SU(1,n)$  in  $\OOO_\C(1,\infty)$ have been also obtained in \cite{Monod18}. \

In rank one, there is in general no hope for an analogue of Margulis superrigidity (even in finite dimension). For example, fundamental groups of non-compact hyperbolic surfaces of finite volume are free groups 
and thus not rigid. For compact hyperbolic surfaces, the lack of rigidity gives rise to the Teichmüller space and thus a whole variety of deformations of the corresponding lattices.

For complex hyperbolic lattices, the complex structure constraints lattices since the Kähler form implies the non-vanishing of the cohomology in degree two.  Furthermore, in finite dimension, the K\"ahler form was succesfully used to define a characteristic invariant that selects representations with surprising rigidity properties, the so-called \emph{Toledo invariant} \cite{Tol,BIW}.

The goal of this paper is to study representations of complex hyperbolic lattices in the groups $\PO_\C(p,\infty)$ and $\PO_\R(2,\infty)$, and the associated isometric actions on the Hermitian symmetric spaces $\calX_\C(p,\infty)$ and $\calX_\R(2,\infty)$.
These spaces  have a Kähler form $\omega$ and this yields a class in bounded cohomology of degree 2 on $G=\PO_\C(p,\infty)$ induced by the cocycle that computes the integral of the K\"ahler form $\omega$ over a straight geodesic triangle $\Delta(g_0x,g_1x,g_2x)$ whose vertices are in the orbit of a basepoint:

$$C_\omega^x(g_0,g_1,g_2)=\frac{1}{\pi}\int_{\Delta(g_0x,g_1x,g_2x)}\omega.$$

We denote by $\kappa^b_G\in\Hb^2(G,\R)$ the associated cohomology class where $G=\PO_\C(p,\infty)$ (see \S\ref{sec:bndbasics}). As in finite dimension, the Gromov norm $\|\kappa^b_G\|_\infty$ is exactly the rank of $\calX_\C(p,\infty)$ (after normalization of the metric). Let $\rho\colon \Gamma\to\PO_\C(p,\infty)$ be a homomorphism of a complex hyperbolic lattice. Pulling back $\kappa^b_G$ by $\rho$, one gets a bounded cohomology class for $\Gamma$ and one can define \emph{maximal representations} of $\Gamma$ as representations maximizing a Toledo number defined as in finite dimension (see Definition \ref{def:max}).

Our main results concern maximal representations of fundamental groups of surfaces and more generally hyperbolic lattices. It is a continuation of previous results for finite dimensional Hermitian targets, see \cite{BI,BIW,Poz,KM} among other references. The meaning of Zariski-density in infinite dimension is explained in the following subsection. For representations with target $\PO_\C(p,\infty)$ we prove rigidity:

%

\begin{thm}\label{thm:hyplat} Let $\G<\SU(1,n)$ be a complex hyperbolic lattice with $n$ a positive integer, and let  $\rho\colon\Gamma\to\PO_\C(p,\infty)$ be a maximal representation. If $p\leq2$ then there is a finite dimensional totally geodesic Hermitian symmetric subspace $\calY\subset\Xx_\C(p,\infty)$ that is invariant by $\Gamma$. Furthermore, the representation $\Gamma\to\Isom(\calY)$ is maximal.

More generally, for any $p\in\N$, there is no maximal Zariski-dense representation $\rho\colon\G\to\PO_\C(p,\infty)$. 
\end{thm}
In particular, since $\calY$ is finite dimensional results of Burger-Iozzi \cite{BI}, the third author \cite{Poz} and Koziarz-Maubon \cite{KM} apply. 

Interestingly enough, the analogous result of Theorem \ref{thm:hyplat}  doesn't hold for the orthogonal group $\OR(2,\infty)$ and $n=1$. Let $\Sigma$ be a compact connected Riemann surface of genus one with one connected boundary component (which is a circle), that is a \emph{one-holed torus}. The fundamental group $\Gamma_\Sigma$ of $\Sigma$ is thus a free group on two generators and a lattice in $\SU(1,1)$. 

\begin{thm}\label{thm:dense}
There are  geometrically dense maximal representations $\rho:\G_\Sigma\to\PO_\R(2,\infty)$.
\end{thm} 
Observe that the properties of maximal representations in $\PO_\R(2,\infty)$ and $\PO_\C(p,\infty)$ are so different because,  for every $p$, the Hermitian Lie group $\OR(2,p)$ is of tube type, while the Hermitian Lie groups $\SU(p,q)$ are of tube type if and only if $p=q$. We refer to Section \ref{sec:tube} for more details. This allows much more flexibility, the chain geometry at infinity being trivial.

Not much is known about representations of complex hyperbolic lattices, and even less so in infinite dimension. In the case of surface groups, instead, from the complementary series of $\PSL_2(\R)$, Delzant and Py exhibited one-parameter families of representations in $\PO_\R(p,\infty)$ for every $p\in\N$ \cite{MR2881312}. Having explicit representations in $\PO_\R(2,\infty)$, it is compelling to determine if they induce maximal representations. Showing that some harmonic equivariant map is actually totally real, we conclude, in Appendix \ref{Appendix} that the Toledo invariant of these representations vanishes.



\begin{remark}The difference between $p\leq2$ and $p>2$ lies in the hypotheses under which we can prove the existence of boundary maps (see \S\ref{subsec:bndmaps}). For $p\leq 2$, we are able to prove the existence of well suited boundary maps under geometric density (a hypothesis we can easily reduce to). Unfortunately, for $p>2$, we can only prove it under Zariski-density which is a stronger assumption.
\end{remark}
\subsection{Boundary maps and standard algebraic groups}\label{subsec:bndmaps}

In order to prove Theorem \ref{thm:hyplat}, we use, as it is now standard in bounded cohomology, boundary maps techniques. Let $\G$ be a lattice in $\SU(1,n)$ and $P$ a strict parabolic subgroup of $\SU(1,n)$. The space $B=\SU(1,n)/P$ is a measurable $\G$-space which is amenable and has very strong ergodic properties, and is thus a \emph{strong boundary} (see Definition \ref{def:strong_boundary}) in the sense of \cite{BFICM}. This space can be identified with the visual boundary of the hyperbolic space $\calX_\C(1,n)$. 

In finite dimension, for example in \cite{Poz}, the target of the boundary map is the Shilov boundary of the symmetric space $\calX_\C(p,q)$. If $p\leq q$, this Shilov boundary can be identified with the space $\calI_p$ of isotropic linear subspaces of dimension $p$ in $\C^{p+q}$. In our infinite dimensional setting, we use the same space $\calI_p$ of isotropic linear subspaces of dimension $p$. 

A main difficulty appears in this infinite dimensional context: this space is not compact anymore  for the natural Grassmann topology. Thus the existence of boundary $\Gamma$-maps $B\to \calI_p$ is more involved than in finite dimension. Such boundary maps have been obtained in a non-locally compact setting when the target is the visual boundary $\partial \calX$ of 	a \cat space $\calX$ of finite telescopic dimension, on which a group $\Gamma$ acts isometrically \cite{Duc13,BDL}. Here, $\calI_p$ is only a closed $G$-orbit of $\partial\calX_\C(p,\infty)$. Actually, $\calI_p$ is a subset of the set of vertices in the spherical building structure on $\partial\calX_\C(p,\infty)$ and the previous result is not sufficient. To prove the existence of boundary maps to $\calI_p$, we reduce to representations whose images are dense, in the sense that is explained below.\\

Following \cite{MR2574741}, we say that  a group $\Gamma$ acting by isometries on a symmetric space (possibly of infinite dimension) of non-positive curvature is \emph{geometrically dense} if there is no strict closed invariant totally geodesic subspace (possibly reduced to a point) nor fixed point in the visual boundary.
For finite dimensional symmetric spaces, geometric density is equivalent to Zariski-density in the isometry group, which is a real algebraic group. To prove Theorem \ref{Shilov}, we rely also on the theory of algebraic groups in infinite dimension introduced in \cite{HK}. Roughly speaking, a subgroup of the group of invertible elements of a Banach algebra is \emph{algebraic} if it is defined by (possibly infinitely many) polynomial equations with a uniform bound on the degrees of the polynomials.

This notion of algebraic groups is too coarse for our goals and we introduce the notion of \emph{standard algebraic groups} in infinite dimension. Let $\calH$ be a Hilbert space and let $\GL(\calH)$ be the group of invertible bounded operators of $\calH$. An algebraic  subgroup of $\GL(\calH)$ is \emph{standard} if it is defined by polynomial equations in the matrix coefficients $g\mapsto\langle ge_i,e_j\rangle$ where $(e_i)$ is some Hilbert base of $\calH$. See Definition \ref{def:stdalggroup}.
With this definition, we are able to show that stabilizers of points in $\partial X_\K(p,\infty)$ are standard algebraic subgroups, and the same holds for stabilizers of proper totally geodesic subspaces. 
\begin{defn}\label{Zdense}A subgroup of $\OO(p,\infty)$ is \emph{Zariski-dense} when it is not contained in a proper standard algebraic group. A representation $\rho\colon\G\to\PO_\K(p,\infty)$ is Zariski-dense if the preimage of $\rho(\G)$ in $\OO(p,\infty)$ is Zariski-dense. For a short discussion about a possible Zariski topology in infinite dimension, we refer to Remark~\ref{rem:zariski}.
\end{defn} 
We show in Proposition \ref{ep} that Zariski-density implies geometric density. 
\begin{proposition}\label{ep}
Let $p\in\N$. Stabilizers of closed totally geodesic subspaces of $\calX_\K(p,\infty)$ and stabilizers of  points in $\partial  \calX_\K(p,\infty)$ are  standard algebraic subgroups of $\OO(p,\infty)$.

In particular, a Zariski-dense subgroup of $\OO(p,\infty)$ is geometrically dense.
 \end{proposition}

\begin{question}\label{qu:1.5}Is it true that the converse of Proposition \ref{ep} holds? Namely, are geometric density and Zariski-density equivalent? 
It is also possible that one needs to strengthen the definition of  standard algebraic groups in order to ensure that geometric density and Zariski-density are the same.
\end{question}

Finally, we get the existence of the desired boundary maps under geometric or Zariski-density. In the following statement, two linear subspaces are said to be transverse if their intersection is trivial.

\begin{theorem}\label{Shilov}Let  $\Gamma$ be a countable group with a strong boundary $B$ and $p\in\N$. 

If $\Gamma$ acts geometrically densely on $\mathcal X_\K(p,\infty)$ with $p\leq 2$,  then there is a measurable $\Gamma$-equivariant map $\phi\colon B\to\Ii_p$. Moreover, for almost all pair $(b,b’)\in B^2$, $\phi(b)$ and $\phi(b’)$ are transverse.

If $\Gamma\to\PO_\K(p,\infty)$ is a representation with a Zariski-dense image,  then there is a measurable $\Gamma$-equivariant map $\phi\colon B\to\Ii_p$. Moreover, for almost all pair $(b,b’)\in B^2$, $\phi(b)$ and $\phi(b’)$ are transverse.

\end{theorem}

\subsection{Geometry of chains} In \cite{MR1509453}, Cartan, introduced a very nice geometry on the boundary $\partial\calX_\C(1,n)$ of the complex hyperbolic space. A \emph{chain} in $\partial\calX_\C(1,n)$ is the boundary of a complex geodesic in $\calX_\C(1,n)$. It is an easy observation that any  two distinct points in $\partial\calX_\C(1,n)$ define a unique chain; moreover, to determine if three points lie in a common chain, one can use a numerical invariant, the so-called \emph{Cartan invariant}. Three points lie in a common chain if and only if they maximize the absolute value of the Cartan invariant. This invariant can be understood as an angle or the oriented area of the associated ideal triangle. See \cite[\S 7.1]{Goldman}.

As in \cite{Poz}, we use a generalization of chains and of the Cartan invariant to prove our rigidity statements. For $p\geq 1$ and $q\in\N\cup\{\infty\}$ with $q\geq p$, we denote by $\Ii_p(p,q)$, or simply $\Ii_p$ if the pair $(p,q)$ is understood, the set of isotropic subspaces of dimension $p$ in $\C^{p+q}$.
A \emph{$p$-chain} (or simply a chain) in $\Ii_p(p,q)$ is the image of a standard embedding of 
$\Ii_p(p,p)$ in $\Ii_p(p,q)$. This corresponds to the choice of a linear subspace of $\C^{p+q}$ where the Hermitian form has signature $(p,p)$. A generalization of the Cartan invariant is realized by the \emph{Bergmann cocycle} $\beta\colon\calI_p^3\to[-p,p]$. Two transverse points in $\calI_p$ determine a unique chain and once again, three points in $\calI_p^3$ that maximize the absolute value of the Bergmann cocycle lie in a common chain.

The strategy of proof of Theorem \ref{thm:hyplat} goes now as follows. We first reduce to geometrically dense representations (Proposition \ref{prop:2}) if needed. Thanks to a now well established  formula in bounded cohomology (Proposition \ref{prop:chainpres}), we prove that a maximal representation of a lattice $\Gamma\leq\SU(1,n)$ in $\OOO_\C(p,\infty)$ has to preserve the chain geometry and almost surely maps 1-chains to $p$-chains (Corollary \ref{chainstochains}).

\subsection*{Outine of the paper}Section \ref{sec:sym} is devoted to the background on Riemannian and Hermitian symmetric spaces in infinite dimension. Section \ref{sec:alggroup} focuses on algebraic and standard algebraic subgroups where Proposition \ref{ep} is proved. Existence of boundary maps is proved in Section \ref{sec:bndth}. In Section \ref{sec:bndbasics}, we provide a short summary of the basic definitions related to maximal representations, and adapt them in infinite dimension. Section \ref{sec:hyplat} deals with representations in $\PO_\C(p,\infty)$, there we prove Theorem \ref{thm:hyplat}, in Section \ref{sec:bndcohomo} we study representations of fundamental groups of surfaces in $\PO_\R(2,\infty)$ and prove Theorem \ref{thm:dense}. The computation of the Toledo invariant for the variation on the complementary series is carried out in Appendix \ref{Appendix}.

\subsection*{Acknowledgements}We would like to thank J.-L. Clerc for explanations about tube type and non-tube type Hermitian symmetric spaces, and their generalizations in infinite dimension; Y. Benoist for discussions about constructions of geometrically dense representations in infinite dimension; J. Maubon for discussion about maximal representations of surface groups and N. Treib for discussion about concrete realizations of the Bergmann cocycle on the Shilov boundary of the group $\PO_\R(2,n)$.

\section{Riemannian  and Hermitian symmetric spaces of infinite dimension}\label{sec:sym}
\subsection{Infinite dimensional symmetric spaces}
In this section, we recall definitions and facts about infinite dimensional Riemannian symmetric spaces. By a Riemannian manifold, we mean a (possibly infinite dimensional) smooth manifold modeled on some real Hilbert space with a smooth Riemannian metric. For a background on infinite dimensional Riemannian manifolds, we refer to \cite{MR1666820} or \cite{MR2243772}. 

\begin{remark}Implicitly, all Hilbert spaces considered in this paper will be separable. In particular, any two Hilbert spaces of infinite dimension over the same field will be isomorphic. The symmetric spaces studied below can be defined as well on non-separable Hilbert spaces but since we will consider representations of countable groups, we can always restrict ourselves to the separable case.
\end{remark}

Let $(M,g)$ be a Riemannian manifold, a \emph{symmetry} at a point $p\in M$ is an involutive isometry  $\sigma_p\colon M\to M$ such that $\sigma_p(p) = p$ and the differential at $p$ is $-\Id$. A \emph{Riemannian symmetric space} is a connected Riemannian manifold such that, at each point, there exists a symmetry. See \cite[\S3]{Duc15} for more details.

We will be interested in infinite dimensional analogs of symmetric spaces of non-compact type. If $(M,g)$ is a symmetric space of non-positive sectional curvature without local Euclidean factor then for any point $p\in M$ the exponential $\exp\colon T_pM\to M$ is a diffeomorphism and, if $d$ is the distance associated to the metric $g$, then $(M,d)$ is a \cat space \cite[Proposition 4.1]{Duc15}. So, such a space $M$ has a very pleasant metric geometry and in particular, it has a visual boundary $\partial M$ at infinity. If $M$ is infinite dimensional then $\partial M$ is not compact for the cone topology.

Let us describe the principal example of infinite dimensional Riemannian symmetric space of non-positive curvature. 

\begin{example}\label{example:GL/O}Let $\calH$ be some real Hilbert space with orthogonal group $\OOO(\mathcal{H})$. We denote by $\LL(\calH)$ the set of bounded operators on $\calH$ and by $\GL(\calH)$ the group of the invertible ones with continuous inverse. If $A\in\LL(\calH)$, we denote its adjoint by $^tA$. An operator $A\in \LL(\calH)$ is \emph{Hilbert-Schmidt} if $\sum_{i,j}\langle Ae_i,Ae_j\rangle^2<\infty$ where $(e_i)$ is some orthonormal basis of $\calH$. We denote by  $\LL^2(\calH)$ the ideal of Hilbert-Schmidt operators  and by $\GL^2(\calH)$ the elements of $\GL(\calH)$ that can be written $\Id+A$ where $A\in\LL^2(\calH)$. This is a subgroup of $\GL(\calH)$: the inverse of $\Id+A$ is $\Id-B$ with $B=A(\Id+A)^{-1}=(\Id+A)^{-1}A\in\LL^2(\calH)$. We also set $\OOO^2(\mathcal{H})=\OOO(\mathcal{H})\cap\GL^2(\mathcal{H})$, and denote by $\SSS^2(\mathcal{H})$ the closed subspace of symmetric operators in  $\LL^2(\calH)$ and by $\PP^2(\mathcal{H})$ the set of symmetric positive definite operators in $\GL^2(\calH)$.

Then $\PP^2(\mathcal{H})$ identifies with the quotient $\GL^2(\calH)/\OOO^2(\mathcal{H})$ under the action of $\GL^2(\calH)$ on $\PP^2(\mathcal{H})$ given by $g\cdot x=gx^tg$ where $g\in\GL^2(\calH)$ and $x\in\PP^2(\calH)$. The space $\PP^2(\mathcal{H})$ is actually a Riemannian manifold, $\GL^2(\calH)$ acts transitively by isometries and the exponential map $\exp\colon\SSS^2(\mathcal{H})\to\PP^2(\mathcal{H})$ is a diffeomorphism. The metric at the origin $o=\Id$ is given by $\langle X,Y\rangle=$Trace$(XY)$ and it has non-positive sectional curvature. Then it is a complete simply-connected Riemannian manifold of non-positive sectional curvature. This is a Riemannian symmetric space and the symmetry at the origin is given by $x\mapsto x^{-1}$. 
\end{example}
A \emph{totally geodesic subspace} of a Riemannian manifold $(M,g)$ is a closed submanifold $N$ such that for any $x\in N$ and $v\in T_xN\setminus\{0\}$, the whole geodesic with direction $v$ is contained in $N$. All the simply connected non-positively curved symmetric spaces that will appear in this paper are totally geodesic subspaces of the space $\PP^2(\mathcal{H})$ described in Example \ref{example:GL/O}.

A \emph{Lie triple system} of $\SSS^2(\mathcal{H})$ is a closed linear subspace $\mathfrak{p}<\SSS^2(\mathcal{H})$ such that for all $X,Y,Z\in\mathfrak{p}$, $[X,[Y,Z]]\in\mathfrak{p}$ where the Lie bracket $[X,Y]$ is simply $XY-YX$. The totally geodesic subspaces $N$ of $\PP^2(\mathcal{H})$ containing $\Id$ are in bijection with the  Lie triple systems $\mathfrak{p}$ of $\SSS^2(\mathcal{H})$. This correspondence is given by $N=\exp(\mathfrak{p})$. See \cite[Proposition III.4]{MR0476820}.

All totally geodesic subspaces of $\PP^2(\mathcal{H})$ are symmetric spaces as well and satisfy a condition of non-positivity of the curvature operator. This condition of non-positivity of the curvature operator allowed a classification of these symmetric spaces  \cite[Theorem 1.8]{Duc15}. In this classification, all the spaces that appear are  the natural analogs of the classical finite dimensional Riemannian symmetric spaces of non-compact type.

The isometry group of a finite dimensional symmetric space is a real algebraic group and thus has a Zariski topology; this is no more available in infinite dimension. Let $(M,g)$ be an irreducible symmetric space of finite dimension and non-positive sectional curvature,  and let $G\leq\Isom(M)$. It is well known that the group $G$ is Zariski-dense if and only if there is no fixed point at infinity nor invariant   totally geodesic  strict subspace (possibly reduced to a point). Thus, following the ideas in \cite{MR2574741}, we say that a group $G$ acting by isometries on a (possibly infinite dimensional) Riemannian symmetric space of non-positive curvature $\calX$ is \emph{geometrically dense} if $G$ has no fixed point in $\partial\calX$ and no invariant closed totally geodesic strict subspace in $\calX$.

\subsection{The Riemannian symmetric spaces $\calX_\K(p,\infty)$}\label{subsec:X(p,q)}
Throughout the paper, $\H$ denotes the division algebra of the quaternions, and $\mathcal {H}$ is a separable Hilbert space over $\K=\R$, $\C$ or $\H$ of infinite dimension. In the latter case the scalar multiplication is understood to be on the right. 
We denote by $\LL(\mathcal{H})$ the algebra  of all bounded $\K$-linear operators of $\mathcal{H}$,  and $\GL(\mathcal{H})$ is the group of all bounded invertible  $\K$-linear operators with bounded inverse. Using the real Hilbert space $\mathcal{H}_\R$  underlying  $\mathcal{H}$, one can consider  $\GL(\mathcal{H})$ as a closed subgroup of $\GL(\mathcal{H}_\R)$. We denote by $A^*$ the adjoint of $A\in\LL(\calH)$. In particular, when $\K=\R$, $A^*= ^t\!A$.

Let $p\in\N$. We fix an orthonormal basis $(e_i)_{i\in\N}$ of the separable Hilbert space $\mathcal{H}$, and we consider the Hermitian form 
$$Q(x)=\sum_{i=1}^p\ov x_ix_i-\sum_{i>p}\ov x_ix_i$$ where  $x=\sum e_ix_i$. The isometry group of this quadratic form will be denoted $\OO(Q)$ or equivalently $\OO(p,\infty)$. 

The intersection of $\OO(p,\infty)$ and the orthogonal group of $\mathcal{H}$ is isomorphic to 
 $\OO(p)\times\OO(\infty)$, where $\OO(p)$ (resp. $\OO(\infty)$) is the orthogonal group of the separable Hilbert space of dimension $p$ (resp. of infinite dimension).
 Then the quotient
$$\calX=\calX_\mathbf K(p,\infty)=\left. \OO(p,\infty)\middle/\left(\OO(p)\times\OO(\infty)\right)\right.$$ 
has a structure of infinite dimensional irreducible Riemannian symmetric space of non-positive curvature. This can be seen via the identification of $\calX_{\mathbf K}(p,\infty)$ with the set 
$$\mathcal{V}=\{V\leq\calH,\ \dim_\K(V)=p,\ Q|_V>0\}.$$  
The group $\OO(p,\infty)$ acts transitively on $\mathcal{V}$ (by Witt’s theorem) and the stabilizer of the span of the $p$ first vectors is $\OO(p)\times\OO(\infty)$. Moreover, the subgroup $\OOO_\K^2(p,\infty)=\OO(p,\infty)\cap\GL^2(\calH)$ acts also transitively on $\mathcal{V}$ and thus $$\calX_{\mathbf K}(p,\infty)\simeq \mathcal{V}\simeq \left.\OOO_\K^2(p,\infty)\middle/\left(\OO(p)\times\OOO_\K^2(\infty)\right)\right. .$$

The stabilizer of the origin in the action of $\OOO_\K^2(p,\infty)$ on $\PP^2(\calH_\R)$ is exactly $\OO(p)\times\OOO_\K^2(\infty)$ and the orbit of $\OOO_\K^2(p,\infty)$ in $\PP^2(\calH_\R)$ is a totally geodesic subspace \cite[Proposition 2.3]{Duc13}. Thus $\calX_\mathbf K(p,\infty)$ has a structure of simply connected non-positively curved Riemannian symmetric space.

Observe that when $\K=\R$ or $\C$, homotheties act trivially on $\calX_\K(p,\infty)$ and thus the group $\PO_\K(p,\infty)$, defined to be $\OOO_\K(p,\infty)/\{\lambda \Id,\ |\lambda|=1\}$, acts by isometries on $\calX_\K(p,\infty)$. Moreover, it is proved in \cite[Theorem 1.5]{Duc13} that this is exactly the isometry group of $\calX_\K(p,\infty)$ when $\K=\R$.


\begin{definition}Let $\calX_1,\calX_2$ be two symmetric spaces of type $\calX_\K(p_i,q_i)$  where $p_i\leq q_i\in\N\cup\{\infty\}$ corresponding to Hilbert spaces $\calH_1,\calH_2$ and Hermitian forms $Q_1,Q_2$. By a \emph{standard embedding}, we mean the data of a linear map $f\colon\calH_1\to\calH_2$ such that $Q_2(f(x),f(y))=Q_1(x,y)$ for all $x,y\in\calH_1$. The  group  $\OOO_\K(Q_1)$ embeds in $\OOO_\K(Q_2)$ in the following way: $f$ intertwines the actions on $\calH_1$ and $f(\calH_1)$ and the action is trivial on the orthogonal of $f(\calH_1)$, which is a supplementary of $f(\calH_1)$ since $Q_2$ is non degenerate on $f(\calH_1)$.

Finally the totally geodesic embedding $\calX_1\hookrightarrow\calX_2$ is given by the orbit of the identity under the action of the orthogonal group of $Q_1$.
\end{definition}

The spaces $\calX_\K(p,\infty)$, with $p$ finite, are very special among infinite dimensional Riemannian symmetric spaces: they have finite rank, which is $p$. This means there are totally geodesic embeddings of $\R^p$ in $\calX_\mathbf K(p,\infty)$ but there are no totally geodesic embeddings of $\R^q$ for $q>p$. Furthermore, every infinite dimensional irreducible Riemannian symmetric space of non-positive curvature operator and finite rank arises this way \cite{Duc15}.

This finite rank property gives some compactness on $\overline{\calX}=\calX\cup\partial \calX$ for a weaker topology \cite[Remark 1.2]{CL}. Moreover we have the following important property.
\begin{proposition}[{\cite[Proposition 2.6]{Duc13}}]\label{prop:finite_config}Any finite configuration of points, geodesics, points at infinity, flat subspaces of $\calX_\mathbf K(p,\infty)$ is contained in some finite dimensional totally geodesic subspace of $\calX_\mathbf K(p,\infty)$ which is a standard embedding of $\calX_\mathbf K(p,q)$ with $q\in\N$.
\end{proposition}

The boundary at infinity $\partial \calX_\mathbf K(p,\infty)$ has a structure of a \emph{spherical building}, which we now recall. We refer to \cite{AbramenkoBrown} for general definitions and facts about buildings, and to \cite{Duc13} for the specific case we are interested in. The space $\partial \calX_\mathbf K(p,\infty)$ has a natural structure of a simplicial complex (of dimension $p-1$): a simplex (of dimension $r-1$) in $\partial\calX_\mathbf K(p,\infty)$ is defined by a flag $(V_1\subsetneq\dots\subsetneq V_r)$, where all the $V_i$ are non-zero totally isotropic subspaces of $\mathcal H$. In particular, vertices of $\partial \calX_\mathbf K(p,\infty)$ correspond to totally isotropic subspaces. A simplex $A$ is contained in a simplex $B$ if all the subspaces appearing in the flag $A$ also appear in the flag $B$. 

Each vertex has a \emph{type}, which is a number between $1$ and $p$ given by the dimension of the corresponding isotropic subspace. 
 More generally, the type of a cell is the finite increasing sequence of dimensions of the isotropic subspaces in the associated isotropic flag.  

\begin{definition}
Two vertices of $\partial\calX_\mathbf K(p,\infty)$, corresponding to isotropic spaces $V$ and $W$, are \emph{opposite} if the restriction of $Q$ to $V+W$ is non-degenerate, which means $W\cap V^\bot=0$. 

Two simplices of the same type, corresponding to two flags $(V_1\subset\dots\subset V_r)$ and $(W_1\subset\dots\subset W_r)$ of the same type are opposite if their vertices of the same type are opposite.
\end{definition}

\begin{remark}
In terms of CAT(0) geometry, we note that two vertices of $\partial\calX_\mathbf K(p,\infty)$ of the same type are opposite if and only if they are joined by a geodesic line in $\calX_\mathbf K(p,\infty)$.
For vertices of dimension $p$, opposition is equivalent to transversality: two vertices $V,W$  with $\dim(V)=\dim(W)=p$ are opposite if and only if $V\cap W=0$.
\end{remark}

\subsection{Hermitian symmetric spaces}\label{sec:Hermitian}

Let $(M,g)$ be a Riemannian manifold (possibly of infinite dimension). An \emph{almost complex structure} is a $(1,1)$-tensor $J$ such that for any vector field $X$, $J(J(X))=-X$.  A triple $(M,g,J)$, where $(M,g)$ is a Riemannian manifold and $J$ is an almost complex structure, is a \emph{Hermitian manifold} if for all vector fields $X,Y$, $g(J(X),J(Y))=g(X,Y)$. If $(M,g,J)$ is a Hermitian manifold, we define a  2-form $\omega$ by the formula $\omega(X,Y)=g(J(X),Y)$.  A \emph{Kähler manifold} is a Hermitian manifold such that $d\omega=0$ and $\omega$ is the \emph{Kähler form} on $M$.

Let $(M,g,J)$ be a Hermitian manifold and $\nabla$ be the Levi-Civita connection associated to the Riemannian metric $g$. The almost-complex structure $J$ is parallel if $\nabla J=0$ that is if for all vector fields $X,Y$, $\nabla_X(JY)=J(\nabla_XY)$. This parallelism condition implies that $\omega$ is parallel as well, that is for all vector fields $X,Y,Z$, $(\nabla_X\omega)(Y,Z)=0$. Since $d\omega (X,Y,Z)=(\nabla_X\omega)(Y,Z)-(\nabla_Y\omega)(X,Z)+(\nabla_Z\omega)(X,Y)$, the parallelism condition $\nabla J=0$ implies that $\omega$ is closed. 

Let  $N$ be the Nijenhuis (2,0)-tensor on $M$, that is for all vector fields $X,Y$,
$$N(X,Y)=2\left([X,Y]-[J(X),J(Y)]-J[J(X),Y]-J[X,J(Y)]\right).$$
The parallelism of $J$ implies that this tensor vanishes.  An almost complex structure $J$ with vanishing Nijenhuis tensor is called a \emph{complex structure}. Thus a parallel almost complex structure is a complex structure.

\begin{definition}Let $(M,g)$ be a simply connected Riemannian symmetric space of non-positive sectional curvature. The symmetric space $M$ is said to be a \emph{Hermitian symmetric space} if it admits a Hermitian almost complex structure $J$ that is invariant under symmetries. This means that for any $p,q\in M$, $$d\sigma_p\circ J_{\sigma_p(q)}\circ d\sigma_p=J_q$$ on the tangent space $T_qM$. One also says that the symmetries are holomorphic. 
\end{definition}

Let us recall a few notations in $\PP^2(\mathcal{H})$. We denote by  $o$ the origin in $\PP^2(\mathcal{H})$, that is the identity $\Id$ of $\calH$. The symmetry $\sigma_o$ at the origin is the map $x\mapsto x^{-1}$. The action $\tau$ of $\GL^2(\calH)$ on $\PP^2(\calH)$ is given by $\tau(g)(x)=gx^tg$. In particular, one has the relation
$$\sigma_o\circ\tau(g)=\tau(g^{-1})\circ\sigma_o.$$
 The exponential map $\exp\colon\LL^2(\calH)\to\GL^2(\calH)$ is a local diffeomorphism around the origin and it induces a diffeomorphism $\exp\colon\SSS^2(\calH)\to\PP^2(\calH)$. In particular, we identify the tangent space at the origin $T_o\PP^2(\calH)$ with the Hilbert space $\SSS^2(\calH)$. The isotropy group of the origin, that is the fixator of $o$,  is $\OOO^2(\calH)$. It acts also on $\SSS^2(\calH)$ by $g\cdot v=gv^tg$ and one has $g\exp(v)^tg=\exp(gv^tg)$ for all $v\in\SSS^2(\calH)$ and $g\in\OOO^2(\calH)$. If $K$ is a subgroup of $\OOO^2(\calH)$, we denote by $K^*$ its image in the isometry group of $\SSS^2(\calH)$. 

The following proposition is a mere extension of a classical statement in finite dimension to our infinite dimensional setting (see for example \cite[Proposition VIII.4.2]{MR1834454}). It can be proved with the same methods.

\begin{proposition}\label{prop:complex_structure}Let $(M,g)$ be a totally geodesic subspace of the symmetric space $\PP^2(\mathcal{H})$ containing $o$ (the identity element) and corresponding to the Lie triple system $\mathfrak{p}$. Let $G$ be the connected component of the stabilizer of $M$ in $\GL^2(\mathcal{H})$ and let $K$ be the isotropy subgroup of $o$ in $G$. Assume there is an operator $J_0\colon \mathfrak{p}\to \mathfrak{p}$ such that
\begin{enumerate}
\item $J_0^2=-\Id$,
\item $J_0$ is an isometry and
\item $J_0$ commutes with all elements of $K^*$ .
\end{enumerate}
Then there is a unique $G$-invariant almost complex structure $J$ on $M$ which coincides with $J_0$ on $T_oM$. Moreover, $J$ is Hermitian and parallel. Thus, $(M,g,J)$ is a Hermitian symmetric space and a Kähler manifold.
\end{proposition}
%

\begin{remark} It is well-known that a finite dimensional manifold with a complex structure $J$ is a complex manifold, that is a manifold modeled on $\C^n$ with holomorphic transition maps. The same result does not hold in full generality for infinite dimensional manifolds but in the case of real analytic Banach manifolds the result still holds (\cite[Theorem 7]{MR2157347}). The Hermitian symmetric spaces we consider have a real analytic complex structure and thus are complex manifolds. Nonetheless, we will not need this result. 
\end{remark}

In the remaining of this section, we exhibit the complex structures $J$ on two classes of Hermitian symmetric spaces we will use later in the paper. Thanks to Proposition~\ref{prop:complex_structure}, it suffices to find $J_0$ with the required properties. The complex structures we describe are all the natural analogs of the corresponding complex structures in finite dimension.

Below, we use orthogonal decompositions $\calH=V\oplus W$ and block decompositions for elements of $\LL(\calH)$. When we write $g=\left[\begin{array}{cc}
A &B\\ C & D
\end{array}\right]$ this means that $A=\pi_V\circ g|_V\in \LL(V)$, $B=\pi_V\circ g|_W\in \LL(W,V)$, $C=\pi_W\circ g|_V\in \LL(V,W)$ and $D=\pi_W\circ g|_W\in \LL(W)$.

\subsubsection*{The Hermitian symmetric space $\calX_\C(p,\infty)$.} Let $\mathcal{H}$ be a complex Hilbert space of infinite dimension. We denote by $\calH_\R$ the underlying real Hilbert space. Let $V,W$ be closed orthogonal complex subspaces of dimension $p\in\N\cup\{\infty\}$ and $\infty$ such that $\calH=V\oplus W$. Let $I_{p,\infty}=\Id_V\oplus-\Id_W$. Thus 
$$\OOO^2_\C(p,\infty)=\{g\in\GL^2(\calH),\ g^*I_{p,\infty}g=I_{p,\infty}\}.$$
The symmetric space $\calX_\C(p,\infty)$ is the $\OOO^2_\C(p,\infty)$-orbit of the identity in $\PP^2(\calH_\R)$, that is the image under the exponential map of  the Lie triple system
$$\mathfrak p=\left\{\left[\begin{array}{cc}0&A\\A^*&0\end{array}\right],\ A\in\LL^2(W,V)\right\}.$$
The complex structure is induced by the endomorphism $J_0$ of $\mathfrak p$ defined by $J_0\left[\begin{array}{cc}0&A\\A^*&0\end{array}\right]=\left[\begin{array}{cc}0&iA\\-iA^*&0\end{array}\right]$.
Since the stabilizer of $\Id_{\calH_\R}$ in $\OOO_\C(p,\infty)$ is given by all the operators that can be expressed as $\left[\begin{array}{cc}P&0\\0&Q\end{array}\right]$ with $P\in\OOO^2_\C(V)$ and $Q\in\OOO^2_\C(W)$, $J_0$ satisfies the conditions of Proposition~\ref{prop:complex_structure}.

\subsubsection*{The Hermitian symmetric space $\calX_\R(2,\infty)$.}
 Let $\calH$ be a real Hilbert space of infinite dimension. Let $V,W$ be closed orthogonal subspaces of dimension $2$ and $\infty$ such that $\calH=V\oplus W$. Let $I_{2,\infty}=\Id_V\oplus-\Id_W$. Thus 
$$\OOO^2_\R(2,\infty)=\{g\in\GL^2(\calH),\ ^tgI_{2,\infty}g=I_{2,\infty}\}.$$
The symmetric space $\calX_\R(2,\infty)$ is the $\OOO^2_\R(2,\infty)$-orbit of the identity in $\PP^2(\calH)$, that is the image under the exponential map of the Lie triple system
$$\mathfrak p=\left\{\left[\begin{array}{cc}0&A\\^tA&0\end{array}\right],\ A\in\LL(W,V)\right\}.$$
Fix some orthonormal basis $(e_1,e_2)$ of $V$ and let $I=\left[\begin{array}{cc}0&-1\\1&0\end{array}\right]$. This element belongs to the group $\SO_\R(2)$, which is commutative. The complex structure is defined by $J_0\left[\begin{array}{cc}0&A\\^tA&0\end{array}\right]=\left[\begin{array}{cc}0&IA\\-^t(IA)&0\end{array}\right]$. Since the stabilizer of $\Id_\calH$ in the identity component of $\OOO^2_\R(2,\infty)$ is given by operators of the form $\left[\begin{array}{cc}P&0\\0&Q\end{array}\right]$ with $P\in\SO_\R(2)$ and $Q\in\OOO^2_\R(\infty)$. Let us denote by $\OOO^{+}_\R(2,\infty)$ the set of elements in $\OOO_\R(2,\infty)$ of the form  $\left[\begin{array}{cc}
A &B\\ C & D
\end{array}\right]$  where $A$ has positive determinant. So there is  a $\OOO^{+}_\R(2,\infty)$-invariant complex structure on $\calX_\R(2,\infty)$. Let us denote by $\PO_\R^+(2,\infty)$ the image of  $\OOO^{+}_\R(2,\infty)$ under the quotient map $\OOO_\R(2,\infty)\to\PO_\R(2,\infty)$.

\subsection{Tube type Hermitian symmetric spaces}\label{sec:tube}

In finite dimension, the class of Hermitian symmetric spaces splits into two classes: those of tube type and those that are not of tube type. This distinction is important for the approach we use to understand maximal representations.  For a definition in finite dimension, we refer to  \cite{BIWtight}.
Let us briefly recall that if $\calX$  is a finite dimensional Hermitian symmetric space, a \emph{chain} is the boundary (as a subset of the Shilov boundary of $\calX$) of a maximal tube type subspace. By definition, if $\calX$ is of tube type, there is a unique maximal tube type subspace: $\calX$ itself. But if $\calX$ is not of tube type, chains lie in a unique $\Isom(\calX)$-orbit and it yields a new incidence geometry: the chain geometry (see \cite[\S3]{Poz}). Let us give an ad hoc definition of tube type Hermitian symmetric spaces in infinite dimension. 

\begin{definition} An irreducible Hermitian symmetric space is \emph{of tube type} if there is  a dense increasing union of tube type finite dimensional totally geodesic holomorphic Hermitian symmetric subspaces. 
\end{definition}

\begin{lemma} The Hermitian symmetric spaces $\calX_\C(\infty,\infty)$, $\calX_\R(2,\infty)$, $\Sp^2(\calH)/\U^2(\calH)$ and the space $\OOO^{*2}(\infty)/\U^2(\infty)$ are of tube type.

The Hermitian symmetric space $\calX_\C(p,\infty)$ with $p<\infty$ is not of tube type.
\end{lemma}

\begin{proof}The four first cases are simply the closure of an increasing union of Hermitian totally geodesic holomorphic subspaces isomorphic to respectively $\calX_\C(n,n),\calX_\R(2,n), \Sp(2n)/\U(n)$ and $\mathrm{SO}^{*}(4n)/\U(2n)$. All those spaces are of tube type.

For $\calX_\C(p,\infty)$ with $p<\infty$, we know that any finite dimensional totally geodesic and holomorphic Hermitian symmetric subspace $Y$ is contained in some standard copy of $\calX_\C(p,q)$ for $q>p$ large enough. In particular, if $Y$ is of tube type then it lies in some standard copy of $\calX_\C(p,p)$ and thus standard copies of $\calX_\C(p,p)$ are maximal finite dimensional Hermitian symmetric subspaces of tube type.\end{proof}

\begin{remark}Among the infinite dimensional Hermitian symmetric spaces of tube type, $\calX_\R(2,\infty)$ is remarkable. This is the only one to be of tube type and of finite rank.
\end{remark}

\begin{remark} A theory of tube type domains and Jordan algebras in infinite dimension has been developed. We refer to \cite{MR0492414} and references for an entrance to this subject. We don’t rely on this theory. 

\end{remark}

\section{Algebraic groups in infinite dimension}\label{sec:alggroup}
\subsection{Algebraic subgroups of bounded operators of Hilbert spaces}
Algebraic subgroups of finite dimensional linear Lie groups are well understood and equipped with a useful topology: the Zariski topology. In infinite dimension some new and baffling phenomena may appear. For example, one parameter subgroups may be non continuous. In \cite{HK}, Harris and Kaup introduced the notion of linear algebraic groups and showed that they behave nicely with respect to the exponential map. In particular, the exponential map is a local homeomorphism and any point sufficiently close to the origin lies in some continuous 1-parameter subgroup.

Let $A,B$ be two real Banach algebras and let $G(A)$ be the set of all invertible elements of $A$. A  map $f\colon A\to B$ is a \emph{homogeneous polynomial map} of degree $n$ if there is a continuous $n$-linear map $f_0\in\LL^n(A,B)$ such that for any $a\in A$, $f(a)=f_0(a,\dots,a)$. Now, a map $f\colon A\to B$ is \emph{polynomial} if it is a finite sum of homogeneous polynomial maps. Its \emph{degree} is the maximum of the degrees that appear in the sum.

Let $\calH$ be a real Hilbert space. The Banach algebras we will use are $\LL(\calH)$ endowed with the operator norm and the field of real numbers $\R$. The group of invertible elements in $\LL(\calH)$ is $\GL(\calH)$.

\begin{definition} A subgroup $G$ of $G(A)$ is an \emph{algebraic subgroup} if there is a constant $n$ and a set $\mathcal{P}$ of polynomial maps of degrees at most $n$ on $A\times A$ such that 
\[G=\left\{g\in G(A);\ P(g,g^{-1})=0,\ \forall P\in \mathcal{P}\right\}.\]
\end{definition} 

Observe that $\mathcal{P}$ may be infinite but the degrees of its elements are uniformly bounded. The main result of  \cite{HK} is the fact that an algebraic subgroup is a Banach Lie group (with respect to the norm topology) and that the exponential map  gives a homeomorphism in a neighborhood of the identity.

\begin{remark}\label{rem:zariski}In this context, one could define a generalized Zariski topology by choosing the smallest topology such that zeros of polynomial maps (or standard polynomial maps, see below) are closed. This topology behaves differently from the finite dimensional case because the Noetherian property does not hold. We will not use any such topology.

	Moreover, the intersection of an infinite number of algebraic subgroups has no reason to be an algebraic group. Degrees of the defining polynomials may be unbounded.
\end{remark}
\begin{examples}\label{examples:algebraicgroups}\begin{enumerate}
\item Let $\calH$ be a Hilbert space of infinite dimension over $\C$ and let $\calH_\R$ be the underlying real Hilbert space. Let $I$ be the multiplication by the complex number $i$. Then $I$ is an isometry of $\calH_\R$ of order 2 and 
$$\GL(\calH)=\{g\in\GL(\calH_\R), gI=Ig\}.$$
Since the map $M\mapsto MI-IM$ is linear on $\LL(\calH)$, $\GL(\calH)$ is an algebraic subgroup of $\GL(\calH_\R)$. Similarly, if $\H$ is the field of quaternions and $\calH$ is a Hilbert space over $\H$ (with underlying real Hilbert space $\calH_\R$) then $\GL(\calH)$ is an algebraic subgroup of $\GL(\calH_\R)$.
\item Let $\calH$ be a Hilbert space of infinite dimension over $\K$ and $\calH=V\oplus W$ be an orthogonal splitting where $V$ has dimension $p\in\N$. Let $I_{p,\infty}$ be the linear map $\Id_V\oplus-\Id_W$. By definition the group $\OOO_\K(p,\infty)$ is 
$$\OOO_\K(p,\infty)=\left\{g\in\GL(\calH),g^*I_{p,\infty}g=I_{p,\infty}\right\}$$
and since the map $(L,M)\mapsto L^*I_{p,\infty}M$ is bilinear on $\LL(\calH_\R)\times\LL(\calH_\R)$, the group $\OOO_\K(p,\infty)$ is a (real) algebraic subgroup of $\GL(\calH_\R)$. This is a particular case of \cite[Example 4]{HK}.
\end{enumerate}
\end{examples}

In finite dimension, linear algebraic groups of $\GL_n(\R)$ are given by polynomial equations in matrix coefficients. We generalize this notion to subgroups of $\GL(\mathcal{H})$.

\begin{definition}\label{def:stdalggroup}Let $\calH$ be a real Hilbert space. A \emph{matrix coefficient} is a linear map $f\colon \LL(\calH)\to\R$ such that there are $x,y\in\calH$ such that $f(L)=\langle Lx,y\rangle$ for any $L\in\LL(\calH)$. A homogeneous polynomial map $P$ on $\LL(\calH)\times\LL(\calH)$ of degree $d$ 
is \emph{standard} if there is an orthonormal basis $(e_n)_{n\in\N}$ of $\calH$, non-negative integers $m,l$ such that $d=m+l$ and families of real coefficients $(\lambda_i)_{i\in\N^{2m}}$, $(\mu_j)_{j\in\N^{2l}}$ such that for all $(M,N)\in\LL(\calH)\times\LL(\calH)$, $P(M,N)$ can be written as an absolutely  convergent series
$$P(M,N)=\sum_{i\in\N^{2m},\ j\in\N^{2l}}\lambda_i\mu_jP_i(M)P_j(N)$$
where $P_i(M)=\prod_{k=0}^{m-1}\langle Me_{i_{2k}},e_{i_{2k+1}}\rangle$ for $i\in\N^{2m}$ and similarly $P_j(N)=\prod_{k=0}^{l-1}\langle Ne_{i_{2k}},e_{i_{2k+1}}\rangle$ for $j\in\N^{2l}$.

A polynomial map is standard if it is a finite sum of standard homogeneous polynomial maps. A subgroup of $\GL(\calH)$ is a \emph{standard algebraic subgroup} if it is an algebraic subgroup defined by a family of standard polynomials.
\end{definition}

\begin{examples}\label{ex:algsub}\begin{enumerate}
\item Any matrix coefficient is a standard homogeneous polynomial map of degree 1. For $x,y\in\calH$ and any orthonormal basis $(e_n)_{n\in\N}$, we define $x_n=\langle x,e_n\rangle$ and similarly $y_n=\langle y,e_n\rangle$ for any $n\in\N$. Then the matrix coefficient $P(M)=\langle Mx,y\rangle$ is given by the series  
$$\sum_{i,j\in\N}x_iy_j\langle Me_i,e_j\rangle.$$
For any finite subset $K\subset \N^2$ finite containing $\{1,\dots,n\}^2$,
$$\left|\sum_{i,j\leq n}x_iy_j\langle Me_i,e_j\rangle-\langle Mx,y\rangle\right|\leq
||M|| \left(\Vert x\Vert  \Vert \pi_n(y)-y\Vert+\Vert y\Vert \Vert \pi_n(x)-x\Vert \right)$$
where $\pi_n(x)$ is the projection on the space spanned by the $n$ first vectors of the basis. This implies that the series is absolutely convergent (see \cite[VII-3-\S8]{MR0253266}).

\item The group $\OOO_\K(p,\infty)$ is not only an algebraic group, it is also a standard algebraic group. Let $(e_i)_{i\in\N}$ be an orthonormal basis of $\calH$ adapted to the decomposition $\calH=V\oplus W$ as in Example \ref{examples:algebraicgroups}. An element $g\in\GL(\calH)$ is in $\OOO_\K(p,\infty)$ if and only if $\langle I_{p,\infty}ge_i,ge_j\rangle=\langle I_{p,\infty}e_i,e_j\rangle$ for any $i,j\in\N$. Since
$\langle I_{p,\infty}ge_i,ge_j\rangle=\sum_{k\in\N}\langle ge_i,I_{p,\infty}e_k\rangle\langle g e_j,e_k\rangle$ and the coefficient $\langle ge_i,I_{p,\infty}e_k\rangle$ is $\varepsilon_k\langle ge_i,e_k\rangle$ with $\varepsilon_k=-1$ for $m\leq p$ and $\varepsilon_k=1$ for $m\geq p$, we see that $\OOO_\K(p,\infty)$  is a standard algebraic group.

\item Let $\calH$ be a real Hilbert space and $V<\calH$ be a closed subspace then $H=\Stab(V)$ is a standard algebraic subgroup of $\GL(\calH)$. Actually,
$$H=\left\{g\in\GL(\calH),\ \langle gx,y\rangle=0,\ \forall x\in V,y\in V^\bot\right\}$$
and thus is a standard algebraic subgroup of $\GL(\calH)$.

It follows that stabilizers of simplices of the building at infinity of $\calX_\K(p,\infty)$ are standard algebraic subgroups of $\OOO_\K(p,\infty)$. Moreover if $\xi$ is a point at infinity, its stabilizer coincides with the stabilizer of the minimal simplex that contains it. See \cite[Proposition 6.1]{Duc13} for details  in the real case. The same argument works as well over $\C$ and $\H$. In particular, stabilizers of points at infinity are standard algebraic subgroups.
\end{enumerate}
\end{examples}

Let $H$ be a standard algebraic subgroup of $\GL(\calH)$. If $E$ is a finite dimensional subspace of $\calH$, we denote by $H_E$ the subgroup of elements $g\in H$ such that $g(E)=E$, $g|_{E^\bot}=\Id$. We identify $H_E$ with an algebraic subgroup of $\GL(E)$. By a \emph{strict algebraic group} $H$, we mean that $H$ is algebraic and $ H\neq \GL(\calH)$.
\begin{lemma}\label{lem:nt} If $H$ is a strict standard algebraic subgroup then there is a finite dimensional subspace $E\subset \calH$ such that $H_E$ is a strict algebraic subgroup of $\GL(E)$.
\end{lemma}
\begin{proof}Let $\mathcal{P}$ be the family of standard polynomials defining the algebraic subgroup $H$. Let $P\in\mathcal{P}$ be a non-constant standard  polynomial. Choose an orthonormal basis $(e_n)$ such that $P$ can be written as an absolutely convergent series. For $n\in\N$, let us set $E_n$ to be the space spanned by $(e_1,\dots,e_n)$. 

In particular for $n$ large enough, the restriction of $P$ to pairs $(g,g^{-1})$ is a non-constant polynomial map on $\GL(E_n)$ and thus defines a strict algebraic subset of $\GL(E_n)$. \end{proof}

\begin{remark} Not all  polynomial maps are standard. The set of compact operators $\LL_c(\calH)$ is closed in $\LL(\calH)$ (it is the closure of the set of finite rank operators) and by Hahn-Banach theorem there is a non-trivial bounded linear form that vanishes on $\LL_c(\calH)$. This linear form is not standard because it vanishes on all finite rank operators.
\end{remark}

\begin{remark}By Proposition~\ref{ep} (proved in \S~\ref{ext}) Zariski-density (Definition~\ref{Zdense}) implies geometric density and we don't know if the converse holds. Lemma~\ref{lem:nt} shows that algebraic subgroups can be tracked by considering finite dimensional subspaces. So, one can think that phenomena similar to those in the finite dimensional case happen and this is maybe a clue that the converse implication between geometric density and Zariski-density holds. In particular, one can show that if $H$ is a strict algebraic subgroup of $\OOO_\K(p,\infty)$ such that there is some finite dimensional subspace $E$ with $H_E\neq\{\Id\}$ then $H$ is not geometrically dense.
\end{remark}

\subsection{Exterior products}\label{ext}
Let $\calH_\K$ be a Hilbert space over $\K$ with Hermitian form $Q$ of signature $(p,\infty)$. We denote by $\calH$ the underlying real Hilbert space and by $(\cdot ,\cdot )$ the real quadratic form $\Re(Q)$. 

The exterior product $\Exterior^k\mathcal{H}$ has a natural structure of pre-Hilbert space and there is a continuous representation $\pi_k\colon\GL(\mathcal{H})\to\GL\left(\Exterior^k\mathcal{H}\right)$ given by the formula $\pi_k(g)(x_1\wedge\dots\wedge x_k)=gx_1\wedge\dots\wedge gx_k$. An orthonormal basis of $\Exterior^k\mathcal{H}$ is given by $(e_{i_1}\wedge\dots\wedge e_{i_k})_{i\in\mathcal{I}}$ where $\mathcal{I}$ is the set of elements $i\in\N^k$ such that $i_1<\dots<i_k$ and $(e_i)$ is an orthonormal basis of $\mathcal{H}$. In other words, if $\langle \cdot,\cdot\rangle$ is the scalar product on $\mathcal H$, then the bilinear form applied to two vectors $x_1\wedge\dots\wedge x_k$ and $y_1\wedge\dots\wedge y_k$ is given by the Gram determinant $\det \left( \langle x_i,y_j\rangle_{i,j=1..k}\right)$. As usual, the completion of $\Exterior^k\mathcal{H}$ is denoted $\overline{\Exterior^k\mathcal{H}}$.

 The space $\Exterior^k\mathcal{H}$ is also endowed with a quadratic form built from $Q$. One defines $(x_1\wedge\dots\wedge x_k,y_1\wedge\dots\wedge y_k)$ to be $\det\left((x_i,y_j)\right)$. As soon as $k\geq2$, this quadratic form is non-degenerate of signature $(\infty,\infty)$ and extends continuously to $\overline{\Exterior^k\mathcal{H}}$. Moreover $\pi_k(\OO(Q))$ preserves this quadratic form.
 
\begin{lemma}\label{lem:basis} Let $(e_i)_{i\in\mathcal I}$ be an orthonormal basis of $\mathcal{H}$ and let $v,w$ be vectors of $\Exterior^k \mathcal{H}$. There are   families $(\lambda_i)$ and $(\mu_i)$ such that for any $g\in\GL(\mathcal{H})$, 
$$(\pi_k(g)v,w)=\sum_{i,j\in\mathcal{I}}\ov{\lambda_i}\mu_j\frac{1}{k!}\sum_{\sigma\in\mathcal{S}_k}\prod_{l=1}^k(ge_{i_l},e_{j_{\sigma(l)}})$$
is a standard polynomial in $g$.
\end{lemma}

\begin{proof} If suffices to write $v=\sum_{i\in\mathcal{I}}\bigwedge_{l=1}^k e_{i_l}\lambda_i$ and $w=\sum_{j\in\mathcal{I}}\bigwedge_{l=1}^k e_{j_l}\mu_j$ and express the scalar product of $\Exterior^k\mathcal{H}$ in the basis $(e_{i_1}\wedge\dots\wedge e_{i_k})_{i\in\mathcal{I}}$. The sum is absolutely convergent because $(\lambda_i)$ and $(\mu_j)$ are Hilbert coordinates. Finally, $(ge_{i_l},e_{j_{\sigma(l)}})=\langle ge_{i_l}, I_{p,\infty}e_{j_{\sigma(l)}}\rangle$ and writing $I_{p,\infty}e_{j_{\sigma(l)}}$ in the Hilbert base, one recovers an absolutely convergent series of matrix coefficients in the Hilbert base $(e_i)$.
\end{proof}

 \begin{lemma}\label{machin} Let $V$ be a non-trivial subspace of $\Exterior^k\mathcal{H}$. The stabilizer of $\overline{V}$ in $\OO(Q)$ is a standard algebraic subgroup.
 \end{lemma}
 
 \begin{proof} If $V$ is a non-trivial subspace of $\Exterior^k\mathcal{H}$, one can choose an orthonormal basis $(v_i)_{i\in I}$ of $\Exterior^k\mathcal{H}$ such that the closure $\overline{V}$ in the Hilbert completion $\overline{\Exterior^k\mathcal{H}}$ is the closed span of $(v_i)_{i\in I_0}$ for some $I_0\subset I$.

 Let $H$ be the subgroup of $\OO(Q)$ stabilizing $\overline{V}$. Thus, by Lemma \ref{lem:basis}, $g$ belongs to $H$ if and only if, for all $ i\in I_0$ and  $j\in I\setminus I_0$, we have $(\pi_k(g)v_i,v_j)=0$.  Thus $H$ is the algebraic subgroup of $\OO(Q)$ defined by the family of polynomials $\mathcal{P}=\{P_{ij}\}$ where $P_{ij}(g)=(\pi_k(g)v_i,v_j)$. 
 \end{proof}

\begin{proof}[Proof of Proposition~\ref{ep}]
We have seen in Example \ref{ex:algsub}  that stabilizers of points at infinity are standard algebraic subgroups. Assume $\calY$ is a strict totally geodesic subspace of $\calX$. Without loss of generality, we assume that $o\in\calY$ and thus $\calY$ corresponds to some Lie triple system $\mathfrak{p}< \SSS^2(\calH)$. Let $\mathfrak{k}=\overline{[\mathfrak{p},\mathfrak{p}]}$ and $\mathfrak m$ be the Lie algebra $\mathfrak{k}\oplus\mathfrak{p}\leq \LL^2(\calH)$. Since $\mathfrak m$ is a Lie algebra, $G=\exp(\mathfrak m)$ is a subgroup of $\GL^2(\calH)$ that is generated by transvections along geodesics in $\calY$. In particular, for any $h\in\OO(p,\infty)$, $h$ normalizes $G$ if and only if $h$ preserves $\calY$. Since $G=\exp(\mathfrak m)$, $h$ normalizes $G$ if and only if it stabilizes $\mathfrak m$ under the adjoint action (i.e. $\Ad(h)(\mathfrak m)=\mathfrak m$, which means that for any $X\in \mathfrak m$, $hXh^{-1}\in \mathfrak m$). Since $\mathfrak m$ is a closed subspace of $\LL^2(\calH)$, we have the splitting $\LL^2(\calH)=\mathfrak m\oplus \mathfrak m^\bot$. So, $h$ stabilizes $\mathfrak m$ if and only if for any $X\in \mathfrak m$ and $Y\in \mathfrak m^\bot$, $\langle hXh^{-1},Y\rangle=0$ where $\langle\ ,\ \rangle$ is the Hilbert-Schmidt scalar product. Finally since the map $(M,N)\mapsto  \langle MXN,Y\rangle$ is bilinear on $\LL(\calH)\times\LL(\calH)$, $H$ is an algebraic subgroup of $\OO(p,\infty)$. It remains to show that these bilinear maps are standard.  Let $(e_n)$ be an orthonormal basis of $\calH$ and let $E_{i,j}=e_i\otimes e_j^*$ be the associated orthonormal basis of $\LL^2(\calH)$, that is $E_{i,j}(x)=\langle x,e_j\rangle e_i$. Thus, let us write $X=\sum_{i,j}X_{i,j}E_{i,j}$ and $Y=\sum_{i,j}X_{i,j}E_{i,j}$ to obtain

$$\langle MXN,Y\rangle=\sum_{i,j,k,l}X_{i,j}Y_{k,l}\langle ME_{i,j}N,E_{k,l}\rangle$$
where 
\begin{align*}
\langle ME_{i,j}N,E_{k,l}\rangle&=\trace\left((ME_{i,j}N)^*E_{k,l}\right)\\
&=\sum_{n\in\N}\langle N^*E_{i,j}M^*E_{k,l}(e_n),e_n\rangle\\
&=\langle N^*E_{j,i}M^*(e_k),e_l\rangle\\
&=\langle E_{j,i}M^*(e_k),N(e_l)\rangle\\
&=\langle e_i,M^*(e_k)\rangle\langle e_j,N(e_l)\rangle\\
&=\langle M(e_i),e_k\rangle\langle e_j,N(e_l)\rangle.
\end{align*}
The absolute convergence of the series can be proven with the same arguments as in Example \ref{ex:algsub}.(1).
\end{proof}


Let $\calH$ be a Hilbert space over $\K$ with a non-degenerate Hermitian form $Q$ of signature $(p,\infty)$ with $p\in\N$. For a finite dimensional non-degenerate subspace $E\subset\calH$ of Witt index $p$, we denote by $\calX_E\subset\calX_\K(p,\infty)$ the subset of  isotropic subspaces of $E$ of dimension $p$. This corresponds to a standard embedding of $\calX_\K(p,q)\hookrightarrow\calX_\K(p,\infty)$ where $(p,q)$ is the signature of the restriction of $Q$ to $E$. Let $\mathcal{E}$ the collection of all such finite dimensional subspaces. We conclude this section with a lemma that shows that the family $(\calX_E)_{E\in\mathcal{E}}$ is cofinal among finite dimensional totally geodesic subspaces. 
\begin{lem}\label{lem:fdtgs}For any finite dimensional totally geodesic subspace $\mathcal{Y}\subset \calX_\K(p,\infty)$, there is $E\in\mathcal{E}$ such that $\mathcal{Y}\subset\calX_E$. 
\end{lem}

\begin{proof}We claim that one can find finitely many points $x_1,\dots,x_n\in\mathcal{Y}$ such that $\mathcal{Y}$ is the smallest totally geodesic subspace of $\calX$ that contains $\{x_1,\dots,x_n\}$. We define by induction points $x_1,\dots,x_k\in\mathcal{Y}$ and $\mathcal{Y}_k$ that is the smallest totally geodesic subspace containing $\{x_1,\dots,x_k\}$. Observe that $\mathcal{Y}_k$ has finite dimension since $\{x_1,\dots,x_k\}\subset\mathcal{Y}$ and $\mathcal{Y}$ has finite dimension. For $x_1$ choose any point in $\mathcal{Y}$ and let $\mathcal{Y}_1$ be $\{x_1\}$. Assume $x_1,\dots,x_k$ have been defined. If $\mathcal{Y}_k\neq\mathcal{Y}$, choose $x_{k+1}\in\mathcal{Y}\setminus\mathcal{Y}_k$. One has $\mathcal{Y}_{k+1}\varsupsetneq\mathcal{Y}_{k}$ and thus $\dim(\mathcal{Y}_{k+1})>\dim(\mathcal{Y}_{k})$. So, in finitely many steps, one gets that there is $n\in\N$ such that $\mathcal{Y}_n=\mathcal{Y}$.

The points $x_1,\dots,x_n$ are positive definite subspaces (with respect to $Q$) of $\mathcal{H}$. So,  let $E$ be the span of these subspaces. Observe that this space has finite dimension and $x_1,\dots,x_n\in\calX_E$. Up to adding finitely many vectors to $E$, we may moreover ensure that $E$ is non-degenerate with Witt index $p$.
\end{proof}
\section{Boundary theory}\label{sec:bndth}

\subsection{Maps from strong boundaries}

Let $G$ be a locally compact, second countable group acting continuously by isometries on $\calX_\K(p,\infty)$ where $p$ is finite. This section is dedicated to the analysis of \emph{Furstenberg maps} also known as \emph{boundary maps} from a measurable boundary of $G$ to the geometric boundary  $\partial\calX_\K(p,\infty)$, or more precisely to some specific part of this boundary. We use a suitable notion of measurable boundary of a group introduced in \cite[\S2]{BFICM}. This definition (see Definition \ref{def:strong_boundary}) is a strengthening of previous versions introduced by Furstenberg \cite{MR0352328} and Burger-Monod \cite{BM2}.

 Let us recall that a Polish space is a topological space which is separable and completely metrizable. By a \emph{Lebesgue $G$-space} we mean a standard Borel space (that is a space and a $\sigma$-algebra given by some Polish space and its Borel $\sigma$-algebra), equipped with a Borel probability measure and an action of $G$ which is measurable and preserves the class of the measure. We denote by $\Prob(\Omega)$ the space of probability measures on a standard Borel space $\Omega$. This is a Polish space for the topology of weak convergence.

\begin{definition}
Let $\Omega$ be a standard Borel space, $\lambda\in\Prob(\Omega)$. Assume that $G$ acts on $\Omega$ with  a measure-class preserving action. The action of $G$ on $\Omega$ is \emph{isometrically ergodic} if for every separable metric space $Z$ equipped with an isometric action of $G$, every $G$-equivariant measurable map $\Omega\to Z$ is essentially constant.
\end{definition}

\begin{remark}
If the action of $G$ on $\Omega$ is isometrically ergodic, then it is ergodic (take $Z=\{0,1\}$ and the trivial action). Furthermore, if the action of $G$ on $\Omega\times \Omega$ is isometrically ergodic, then it is also the case for the action on $\Omega$. 
\end{remark}

\begin{definition}
Let $Y$ and $Z$ be two Borel $G$-spaces and $p:Y\to Z$ be a Borel $G$-equivariant map. We denote by $Y\times_p Y$ the \emph{fiber product over} $p$, that is the subset $\{(x,y)\in Y^2,\ p(x)=p(y)\}$ with its Borel structure coming form $Y^2$.

We say that $p$ (or $Y$) admits a \emph{fiberwise isometric action} if there exists a Borel, $G$-invariant map $d:Y\times_p Y\to \R$ such that any fiber $Y’\subset Y$ of $p$ endowed with  $d|_{Y’\times Y’}$  is a separable metric space. \end{definition}

Before going on, let us give a few examples of fiberwise isometric actions. These examples are closed to measurable fields of metric spaces that appear in \cite{BDL} and are simpler versions of fiberwise isometric actions that will appear in the proof of Theorem \ref{Shilov}. 

\begin{example}\label{ex:spheres}
 Let $(M,d)$ be a metric space. The \emph{Wisjman hyperspace} $2^M$ is the set of closed subspaces in  $M$. This space can embedded in the space $C(M)$ of continuous functions on $M$: to any close subspace $A$, one associates the distance function $x\mapsto d(A,x)$. The topology of pointwise convergence on $C(M)$ induces the so-called \emph{Wisjman topology} on $2^M$. If $(M,d)$ is complete and separable then the Wisjman hyperspace is a Polish space. Actually, when $M$ is separable, the topology is the same as the topology of pointwise convergence on a countable dense subset.

Let $\calX$ be a complete separable CAT(0) space. We denote by $\mathcal{F}_k$ the space of flat subspaces of dimension $k$ in $\calX$ (i.e. isometric copies of $\R^k$). One can check that $\mathcal{F}_k$ is closed in $2^{\cal X}$:  flatness is encoded in three conditions (equality in the CAT(0) inequality, convexity and geodesic completeness), the dimension is encoded in the Jung inequality (see e.g. \cite{MR1456512}) and all these conditions are closed. The visual boundary $\partial \calX$ with the cone topology is a closed subspace of $\overline{\calX}$ which is an inverse limit of a countable family of closed balls \cite{BH}. Thus $\partial \calX$ is a Polish space.

Let $G$ act by isometries on $\calX$, and let $k>0$ be such that there exists a $k$-dimensional flat in $\calX$. Let  $\partial\mathcal F_k=\{(F,\xi)\mid F\in\mathcal F_k\textrm{ and } \xi\in\partial F\}$. This is a  closed subspace of $\mathcal{F}_k\times \partial \calX$. Then the continuous projection $\partial \mathcal F_k\to\mathcal F_k$ admits a fiberwise isometric action of $\Isom(\cal X)$, each $\partial F$ being endowed with the Tits metric. 

\end{example}

%
%
%
%

\begin{definition}\label{def:relative_ergodic} 
Let $A$ and $B$ be Lebesgue $G$-spaces. Let $\pi:A\to B$ be a measurable $G$-equivariant map. We say that $\pi$ is \emph{relatively isometrically ergodic} if each time we have a $G$-equivariant Borel map $p:Y\to Z$ of standard Borel $G$-spaces which admits a fiberwise isometric action, and measurable $G$-maps $A\to Y$ and $B\to Z$ such that the following diagram commutes
$$
\xymatrix{
    A  \ar[d]_\pi\ar[r]  & Y \ar[d]^p \\
    B \ar[r] & Z
  }
$$
then there exists a measurable $G$-map $\phi\colon B\to Y$ which makes the following diagram commutative.
$$
\xymatrix{
    A  \ar[d]_\pi\ar[r]  & Y \ar[d]^p \\
    B \ar[r]\ar@{.>}[ur]^\phi & Z
  }
$$
\end{definition}

\begin{remark}
If a $G$-Lebesgue space $B$ is such that the first projection $\pi_1\colon B\times B\to B$ is relatively isometrically ergodic, then it is isometrically ergodic. Indeed if $Y$ is a separable metric $G$-space and $f:B\to Y$ is $G$-equivariant then it suffices to apply relatively isometric ergodicity to the map $\tilde{f}\colon(b,b’)\mapsto f(b')$ and the trivial fibration $Y\to\{\ast\}$.
$$
\xymatrix{
    B\times B  \ar[d]_{\pi_1}\ar[r]^{\tilde{f}}  & Y \ar[d] \\
    B \ar[r]\ar@{.>}[ur]^\phi & \{\ast\}
  }
$$
Actually, relative isometric ergodicity yields a measurable map $\phi\colon B\to Y$ such that for almost all $(b,b’)$, $\phi(b)=f(b’)$ and thus $f$ is essentially constant.
\end{remark}

Let $B$ be a Lebesgue $G$-space. We use the definition of amenability for actions introduced by Zimmer, see \cite[\S4.3]{ZimmerBook}. The action $G\action B$ is \emph{amenable} if for any compact metrizable space $M$ on which $G$ acts continuously by homeomorphisms there is a measurable $G$-equivariant map $\phi\colon B\to\Prob(M)$. 

\begin{definition}\label{def:strong_boundary}
The Lebesgue $G$-space $B$ is a \emph{strong boundary} of $G$ if
\begin{itemize}
\item the action of $G$ on $(B,\nu)$ is amenable (in the sense of Zimmer) and
\item the first projection $\pi_1\colon B\times B\to B$ is relatively isometrically ergodic.
\end{itemize}
\end{definition}

\begin{example}
The most important example for us is the following \cite[Theorem 2.5]{BFICM}. Let $G$ be a connected semisimple Lie group, and $P$ a minimal parabolic subgroup. Then $G/P$, with the Lebesgue measure class, is a strong boundary for the action of $G$. If $\Gamma<G$ is a lattice, then $G/P$ is also a strong boundary for the action of $\Gamma$. More generally, this is also true if $G$ is a semisimple algebraic group over a local field.
\end{example}

The next example shows that every countable group admits a strong boundary. 

\begin{example}
Let $\Gamma$ be a countable group, and $\mu\in \Prob(\Gamma)$ be a symmetric measure whose support generates $\Gamma$. Let $(B,\nu)$ be the Poisson-Furstenberg boundary associated to $(\Gamma,\mu)$. Then $B$ is a strong boundary of $\Gamma$ \cite[Theorem 2.7]{BFICM}.
\end{example}

Existence of Furstenberg maps is already known. In the next section, we show that we can specify the type of points in the essential image and get that these points are essentially opposite. Let us recall that an isometric action of a group $\Gamma$ on a \cat space $\calX$ is \emph{non-elementary} if there is no invariant flat subspace (possibly reduced to a point) nor a global fixed point at infinity.

\begin{theorem}[{\cite[Theorem 1.7]{Duc13}}]\label{thm:duc}
Let $\G$ be locally compact second countable group, $B$ a  strong boundary for $G$ and $p\in\N$. For
any continuous and non-elementary action of $\G$ on $\calX_\K(p,\infty)$ there exists a measurable $\G$-map
$\phi : B \to \partial\calX_\K(p,\infty) $.
\end{theorem}

We will also rely on results obtained in \cite{BDL}. Unfortunately, \cite{BDL} was written before the final version of \cite{BFICM} and a slightly different language was used there. Group actions on measurable metric fields were used there and here we just described fiberwise isometric actions. In the following proposition, we establish the relation between these two notions. We refer to \cite[\S3]{BDL} for definitions, notations and a discussion about measurable metric fields. Roughly speaking, a measurable metric field over a Lebesgue $\Gamma$-space $\Omega$ is a collection $\mathbf{X}=(X_\omega)_{\omega\in\Omega}$ of metric spaces $(X_\omega,d_\omega)$ where one moves from one such metric space to another one in a measurable fashion. It admits a $\Gamma$-action if for all $\omega\in\Omega, g\in\Gamma$, there is an isometry $\sigma(g,\omega)\colon X_\omega\to X_{g\omega}$  that satifies the cocycle relation $\sigma(gg’,\omega)=\sigma(g,g’\omega)\circ\sigma(g’,\omega)$ almost surely.

\begin{lemma}\label{lem:field_fiber} Let $\Gamma$ be a countable group and let $\mathbf{X}$ be a measurable metric field over a Lebesgue $\Gamma$-space $\Omega$ with a $\Gamma$-action. 
Then there is a $\Gamma$-invariant Borel subset $\Omega_0\subset\Omega$ of full measure,  a standard Borel structure on $X=\sqcup_{\omega\in\Omega_0}X_\omega$ and a Borel map $ p \colon X\to\Omega_0$ such that $ p $ admits a $\Gamma$-fiberwise isometric action. Moreover, the fiber $ p ^{-1}(\omega)$ is $X_\omega$ with the metric $d_\omega$.

If $x$ is an invariant section of $\mathbf{X}$ then $x$ corresponds canonically to a $\Gamma$-equivariant measurable map $\Omega_0\to X$.
\end{lemma}

\begin{proof} Let $\{x^n\}_{n\in\N}$ be a fundamental family for the field $\mathbf{X}$. One can find a Borel subset $\Omega_0\subset \Omega$ of full measure such that all the maps $\omega\mapsto d_\omega(x_\omega^n,x^m_\omega)$ are Borel for all $n,m\in\N$ on $\Omega_0$. Observe that $\Omega_0$ is a Lebesgue space as well \cite[\S12.B]{MR1321597}. Up to replace $\Omega_0$ by $\cap_{\gamma\in\Gamma}\gamma\Omega_0$, we may assume that $\Omega_0$ is $\Gamma$-invariant and still a Lebesgue space. 

Let us set $X=\bigsqcup_{\omega\in\Omega_0}X_\omega$ and define $ p \colon X\to \Omega_0$ such that $ p (x)$ is the unique $\omega\in\Omega_0$ with $x\in X_\omega$. Let us define $\phi_n\colon X\to\R$ by the formula $\phi_n(x)= d_{ p (x)}\left(x,x^n_{ p (x)}\right)$. Now, let $\mathcal{A}$ be the smallest $\sigma$-algebra such that $ p $ and $\phi_n$ are measurable for all $n$. To show that $(X,\mathcal{A})$ is a standard Borel space, it suffices to show that $\mathcal{A}$ is countably generated and separates points \cite[\S12.B]{MR1321597}. It is countably generated because $\Omega_0$ and $\R$ are so. Let  $x\neq y\in X$. If $ p (x)\neq p (y)$ then there is a Borel subset $\Omega’\subset\Omega_0$ such that $ p (x)\in\Omega’$ and $ p (y)\notin\Omega’$ thus $ p ^{-1}(\Omega’)\in\mathcal{A}$ separates $x$ and $y$. If $ p (x)= p (y)=\omega$ then by density of $\{x^n_\omega\}$ in $X_\omega$, there is $n$ such that $\phi_n(x)<\phi_n(y)$ and thus $\mathcal{A}$ separates $x$ and $y$.
Moreover, for $(x,y)\in X\times_ p  X$, we simply note $d(x,y)$ for $d_{p(x)}(x,y)$. Then, $d(x,y)=\sup_{n\in\N}|\phi_n(x)-\phi_n(y)|$ and thus $d$ is a Borel map. Since $\Omega_0$ is $\Gamma$-invariant, $ p \colon X\to \Omega_0$ admits a fiberwise isometric $\Gamma$-action.

If $x$ is a section of $\mathbf{X}$, that is an element of $\Pi_{\omega\in\Omega}X_\omega$ with measurability conditions \cite[Definitions 8 and 9]{BDL}, let us use the same notation for the map $x\colon \Omega_0\to X$ such that $x(\omega)=x_\omega$. By construction $\pi\circ x$ and $\phi_n\circ x$ are measurable and thus $x$ is measurable. If the section is invariant then this yields equivariance of the map $x\colon \Omega_0\to X$.
\end{proof}

\begin{remark}With this lemma, for any two Lebesgue $\Gamma$-spaces $A, B$ with a $\Gamma$-factor map, that is a measurable surjective $\Gamma$-map $\pi\colon A\to B$, relative isometric ergodicity of $\pi$ as stated in \cite[Definition 25]{BDL} follows from Definition \ref{def:relative_ergodic} above. Actually if $\mathbf{X}$ is a metric field over $B$, and $B_0$, $p\colon X\to B_0$ are given by Lemma \ref{lem:field_fiber}, this relative ergodicity is reflected in the following diagram where $A_0=\pi^{-1}(B_0)$.
$$
\xymatrix{
    A _0 \ar[d]_\pi\ar[r]  & X \ar[d]^p \\
    B_0 \ar[r]_{\Id}\ar@{.>}[ur] & B_0
  }
$$

This allows us to use freely the results from \cite{BDL}.
\end{remark}
 
\subsection{Equivariant maps to the set of maximal isotropic subspaces under Zariski-density}\label{sb}
As before, let $\calH$ be a Hilbert space over $\K$ with a Hermitian form $Q$ of signature $(p,q)$ with $p<q$ and $q\in\N\cup\{\infty\}$. We fix some locally compact second countable group $G$ with a continuous action by isometries on $\calX_\K(p,q)$ and strong boundary $B$. 

We denote by $\mathcal{I}_k$  the space of totally isotropic subspaces of $\mathcal{H}$ of dimension $k\leq p$. Following the end of Section \ref{subsec:X(p,q)}, this space can be identified with a type of vertices of the spherical  building structure on $\partial\calX_{p,q}$. When recalling the signature of the Hermitian form will seem to help comprehension we will include it in our notation, and denote the space of totally isotropic subspaces as $\mathcal I_k(p,q)$. Let us observe that if $p,q$ are finite then $\mathcal I_k(p,q)$ can be identified with some homogeneous space $G/P$ where $G=\OO(p,q)$ and $P$ is a parabolic subgroup. In that case, we endow $G/P$ with the corresponding $\sigma$-algebra and the unique $G$-invariant measure class on it (see e.g. \cite[Appendix B]{MR2415834}).  Observe that when $p=k=1$ (and $q$ is finite or infinite) then $\mathcal I_1(1,q)$ is merely the visual boundary $\partial \calX_\K (1,q)$ of the hyperbolic space $\calX_\K (1,q)$ of dimension $q$ over $\K$.  For the application to maximal representations, the space $\mathcal{I}_p$ of totally isotropic subspaces of maximal dimension plays an important role.

For example, if $\mathcal H$ is a finite dimensional complex vector space, $\mathcal X_\C (p,q)$ is a complex manifold admitting a bounded domain realization whose Shilov boundary  can be $\SU(p,q)$-equivariantly identified with $\mathcal{I}_p$.  
The purpose of this section is to associate to geometrically dense or Zariski-dense representations $\rho$, an equivariant boundary map with values in the set of maximal isotropic subspaces $\mathcal{I}_p$ (Theorem \ref{Shilov}). 

\begin{remark} Under the hypothesis of Theorem \ref{thm:duc}, we get maps $B\to \mathcal{I}_k$ for at least one $k$: indeed, assume that $\phi$ is a map obtained by Theorem \ref{thm:duc}. Considering the smallest cell of the spherical building at infinity containing $\phi(b)$, one gets a map $B\to\mathcal{F}$ where $\mathcal{F}$ is a space of totally isotropic flags of $\mathcal{H}$ (see Section 6 in \cite{Duc13}). Note that by ergodicity the type of this flag is constant. Thus for each dimension $k$ that appears in this flag, one gets a map $B\to \mathcal{I}_k$. 
\end{remark}

First, we prove opposition for boundary maps to $\mathcal{I}_k$ under Zariski-density.

\begin{proposition}\label{cor} Let $k\leq p$ and assume that $\G$ is countable. Assume that the action $\G\action \calX_\K(p,q)$ is Zariski-dense. If $\phi\colon B\to \mathcal{I}_k$ is a $\G$-equivariant measurable map then for almost every $(b,b’)\in B\times B$, $\phi(b)$ is opposite to $\phi(b')$.
\end{proposition}

\begin{proof}

We  denote by $V_b$ the linear subspace of dimension $k$ corresponding to  $\phi(b)$ and let $\ell_b$ be the corresponding line in $\Exterior^k\mathcal{H}$. 
  By ergodicity of the action $\G\action B\times B$, one of the three following cases happens for almost all $(b,b')$:
  \begin{itemize}
  \item either $\ell_b=\ell_{b'}$ (which means that $V_b=V_{b'}$),
  \item $\ell_b$ and $\ell_{b'}$ span a totally isotropic plane in $\Exterior^k\mathcal{H}$ (in other words $0\neq V_b\cap V_{b'}^\bot\neq 
V_b$) 
\item or $\ell_b$ and $\ell_{b'}$ span a non-degenerate plane, that is $V_b\cap V_{b'}^\bot=\{0\}$. In other words $V_b$ and $V_{b’}$ are opposite.
\end{itemize}
Our goal is to show that only the third case can happen. Assume first that $V_b=V_{b'}$ for almost every $(b,b')$. Then the map $b\mapsto V_b$ is essentially constant and its essential image is a $\G$-invariant vertex. This contradicts the assumption that $\G$ does not fix a point in $\partial\calX_\K(p,\infty)$.
 
 Now assume that for almost every $(b,b')$ the lines
 $\ell_b$ and $\ell_{b'}$ are orthogonal, namely their span is an isotropic plane. Then, thanks to Fubini's theorem, there exists $b\in B$ and $B_b\subset B$ with full measure such that, for any $b'\in B_b$, $\ell_{b'}$ is orthogonal to $\ell_b$. Let $B'$ be the intersection $\cap_{\gamma\in \G}\gamma B_b$. The set $B'$ has full measure and is $\G$-invariant, thus the space spanned by $\left\{\ell_{b’},\ b’\in B’\right\}$ is a proper subspace (being included in the orthogonal of $\ell_b$) and is $\G$-invariant. The closure of this space is not $\OO(p,\infty)$-invariant since this group acts transitively on the space of totally isotropic subspaces of dimension $k$. We conclude by using Lemma \ref{machin}. 
\end{proof}

In the remaining of this section, our goal is to show existence of maps from a strong boundary $B$ to $\mathcal{I}_p$. We begin our discussion by observing that for every $k$ there is a natural fiberwise isometric action of $\G$ over $\mathcal{I}_k$: we denote by $\mathcal V_k$ the space of subspaces $V$ of $\mathcal{H}$ with dimension $p$ such that $Q|_V$ is non-negative and $\ker(Q|_V)$ has dimension $k$. We endow both $\mathcal I_k$ and $\mathcal V_k$ with the induced topologies coming from the corresponding Grassmannians $\mathcal{G}_k,\mathcal{G}_p$ of subspaces of dimension $k$ and $p$ in $\mathcal{H}$. Let us recall that a  complete and separable distance on the Grassmannian $\mathcal{G}_m$ of all subspaces of dimension $m$ is given by $$d(V,W)^2=\sum_{i=1}^m\alpha_i^2$$ where $\alpha_1,\dots,\alpha_m$ are the principal angles between $V$ and $W\in\mathcal{G}_m$. This topology also coincides with the Wisjman topology of the hyperspace $2^\calH$.

The natural projection 
$$\begin{array}{rcl}
\pi\colon \mathcal V_k&\to& \mathcal{I}_k\\
V&\mapsto &\ker(Q|_V)
\end{array}$$
is continuous. For $V_0\in\mathcal{I}_k$, the fiber  $\pi^{-1}(V_0)$ can be identified with a symmetric space $\mathcal X(V_0^\bot/V_0,Q)$ that we define in the following lines. The kernel of $Q$ restricted to $V_0^\bot$ is exactly $V_0$ and thus $Q$ defines a strongly non-degenerate Hermitian form on $V_0^\bot/V_0$ of signature  $(p-k,q)$. So, we define  $\mathcal X(V_0^\bot/V_0,Q)$
to be the symmetric space associated to that Hermitian form, that is the collection positive subspaces of $V_0^\bot/V_0$ of dimension $p-k$. The metric on $\mathcal X(V_0^\bot/V_0,Q)$ is given by the hyperbolic principal angles \cite[\S 3.1]{Duc13}. The preimages of such a positive subspace under the projection $V_0^\bot\to V_0^\bot/V_0$ are in bijective correspondence with the elements in the fiber of $\mathcal V_k$ above $V_0$. 

Recall that $V_0,W_0\in\mathcal{I}_k$ are opposite if the restriction of $Q$ to $V_0+W_0$ is non-degenerate and thus has signature $(k,k)$. If $V_0,W_0$ are opposite then $\mathcal{H}=V_0\oplus W_0^\bot$ because $W_0^\bot$ has codimension $k$ and $V_0\cap W_0^\bot=\{0\}$. So, there is  a bijective correspondence $\sigma_{V_0,W_0}\colon \mathcal X(V_0^\bot/V_0,Q)\to \mathcal X(W_0^\bot/W_0,Q)$ given by the formula
$$\sigma_{V_0,W_0}(V)=(V\cap W_0^\bot)+W_0$$
for $V\in\pi^{-1}(V_0)$. This map is well-defined  because $V\cap W_0^\bot$ has dimension $p-k$ and is positive definite for $Q$. The inverse is given by 

$$\sigma^{-1}_{V_0,W_0}(W)=\sigma_{W_0,V_0}(W)=(W\cap V_0^\bot)+V_0.$$

\begin{lemma}If $V_0,W_0\in\mathcal{I}_k$ are opposite then the map $\sigma_{V_0,W_0}\colon  \mathcal X(V_0^\bot/V_0,Q)\to \mathcal X(W_0^\bot/W_0,Q)$ is an isometry.\end{lemma}

\begin{proof}Since $V_0$ and $W_0$ are opposite, we have the following orthogonal decomposition
$$\mathcal{H}=\left(V_0\oplus W_0\right)\oplus^\bot\left(V_0^\bot\cap W_0^\bot\right)$$
and the restriction of $Q$ to $V_0^\bot\cap W_0^\bot$ is non-degenerate of signature $(p-k,\infty)$. In particular, $V_0^\bot=V_0\oplus (V_0^\bot\cap W_0^\bot)$ and thus the quotient map induces an isomorphism $(V_0^\bot/V_0,Q)\simeq (V_0^\bot\cap W_0^\bot,Q)$ as spaces with Hermitian forms.

Now, if $V\in\pi^{-1}(V_0)$ is written $V=V_0+V'$ where $V'=V\cap W_0^\bot$ then $V'\subset V_0^\bot\cap W_0^\bot$. In particular $\sigma_{W_0,V_0}(V)=V'+W_0$. By construction of the metric via hyperbolic principal angles, the following map are isometries
$$\begin{array}{ccccc}
\mathcal{X}(V_0^\bot/V_0,Q)&\leftarrow&\mathcal{X}(V_0^\bot\cap W_0^\bot,Q)&\to&\mathcal{X}(W_0^\bot/W_0,Q)\\
V_0+V'&\mapsfrom&V'&\mapsto&W_0+V'
\end{array}.$$
Finally, $\sigma_{W_0,V_0}$ is an isometry being the composition of two isometries.\end{proof}

In particular, $\sigma_{V_0,W_0}$ maps flat subspaces to flat subspaces.

\begin{proof}[Proof of Theorem \ref{Shilov} in the Zariski-dense case] In this proof, we freely use measurable metric fields thanks to Lemma \ref{lem:field_fiber}.
We know from Theorem \ref{thm:duc} that there exists a $\G$-map to the visual boundary $\partial\mathcal X(p,\infty)$ and, by ergodicity, we get a $\G$-equivariant map $\phi\colon B\to\mathcal{I}_k$ for some $k\geq1$. Assume that $k$ is a maximal such integer.  If $k=p$ we are done. Assume then that $k<p$. Denoting $V_b\in\mathcal{I}_k$ for $\phi(b)$ and $X_b=\calX(V_b^\bot/V_b,Q)$, we obtain a measurable field of (non-trivial) CAT(0) spaces $\mathbf{X}=\{X_b\}$ with a $\Gamma$-isometric action. Thanks to \cite[Theorem 1.8]{Duc13}, either there is a an invariant section of the metric field $\mathbf{\partial X}=\{\partial X_b\}$ or there is a $\Gamma$-equivariant Euclidean subfield $\mathbf{F}=\{F_b\}$ with $F_b\subset {X}_b$ for almost all $b\in B$.  

In the first case, the stated result easily follows: to every point in $\partial X_b$, one can associate a (non-trivial) totally isotropic flag in $V_b^\bot/V_b$ and we can choose the totally isotropic subspace of maximal dimension in such a flag (whose dimension is essentially constant by ergodicity) to define a totally isotropic subspace $V^\prime_b$ of $\mathcal{H}$ strictly containing  $V_b$. Thus, we get a $\Gamma$-map $\phi^\prime\colon B\to \mathcal{I}_{k^\prime}$ for $k^\prime>k$, contradicting the maximality of $k$.

In order to conclude the proof it is enough to show that, under our hypotheses, there cannot exist a $\Gamma$-equivariant Euclidean subfield. So let us assume that there exists such a subfield. In other words, we  have  a $\Gamma$-map $\psi_0\colon B\to F$ where $F$ is the Polish space constructed from $\mathbf{F}$ thanks to Lemma \ref{lem:field_fiber}, such that $\psi_0(b)$ is a flat in $X_b$. Let us merely denote $F_b$ for $\psi_0(b)$. Note that the map $b\mapsto \dim(F_b)$ is measurable (\cite[Lemma 14]{BDL}), hence $F_b$ is essentially of constant dimension. Among all possible such maps $\psi_0$, we choose one such that this dimension, say $k_0$, is minimal.

 Let us denote $\sigma_{b,b^\prime}=\sigma_{\phi(b),\phi(b')}$ the correspondence isometry from $ X_b$ to $ X_{b^\prime}$ defined above. 

We claim first that $\sigma_{b’,b}\left(F_{b^\prime}\right)$ is parallel to $F_b$. The proof of this statement is very similar to the one of \cite[Theorem 34]{BDL}. We will explain the proof quickly and refer to \cite{BDL} for more details (in particular about measurability of the various maps which appear during the proof).

 Consider the function $f_{b,b’}$ defined on $F_b$ by  $f_{b,b’}(x)=d(x,\sigma_{b’,b}(F_{b^\prime}))$ (recall that both $F_b$ and $\sigma_{b’,b}(F_{b^\prime})$ are flat subspaces of $X_b$). Then $f_{b,b’}$ is a convex function on the Euclidean space $F_b$.  Using Proposition 4 from \cite{BDL}, we see that 4 cases are possible for $f_{b,b’}$, which are described below. By the arguments from the proof of \cite[Theorem 34]{BDL}, these four conditions are measurable and $\G$-invariant, so that one of them happens almost surely.

The first case is when $f_{b,b’}$ does not attain its infimum $m$. In that case, one can consider the sequence $E_n^{b,b'}$ of subset of $F_b$ defined as $E_n^{b,b'}=\{x\mid f(x)\leq m+1/n\}$. By \cite[Proposition 8.10]{Duc13} this sequence of subsets gives a $\Gamma$-map $\xi:B\times B\to \partial F$ where $\partial F$ is the Borel space associated to the metric field $\partial \mathbf{F}$, such that $\xi(b,b’)\in \partial F_b$ for almost all $(b,b’)$. Since $\partial \mathbf{F}$ a metric field with a $\Gamma$-action, using relative isometric ergodicity, we see that $\xi$ does not in fact depend on $b’$, and therefore we have a map $\xi: B\to \partial X$ (where $\partial X$ is the Borel space associated to the metric field $\partial \mathbf X$) such that $\xi(b)\in\partial X_b$. Now $X_b=\calX(V_b^\bot/V_b,Q)$ has a boundary which is a spherical building where cells correspond to totally isotropic flags in $V_b^\bot/V_b$. Therefore to a point in the boundary one can associate a totally isotropic subspace $W\subset V_b^\bot/V_b$, which we can lift to a totally isotropic space $\overline{W}$ containing $V_b$ in $\mathcal H$. Thus the map $\xi$ gives rise to a map $B\to \mathcal I_{k'}$ with $k'>k$, contradicting the assumption on $k$.

If $f_{b,b’}$ attains its minimum $m$, let $Y=f_{b,b’}^{-1}(m)$, which depends on $b$ and $b'$.  
The second case is when $Y$ is bounded. Then one can consider its circumcenter $y(b,b')$. The map $(b,b')\mapsto y(b,b')$ is measurable \cite[Lemma 8.7]{Duc13} and $\Gamma$-equivariant. By relative isometric ergodicity, it does not in fact depend on $b'$. So there is a $\Gamma$-map $x\colon B\to X$ such that $x(b)\in X_b$. In particular, $\{x(b)\}$ is a Euclidean subfield of $X_b$ and by minimality  of $k_0$, $F_b=\{x_b\}$ and thus we get parallelism of $\sigma_{b,b’}(F_{b’})$ and $F_b$ because any two points are parallel as Euclidean subspaces of dimension 0.

In the third case, one can write $Y=E\times T$ where $E$ is subflat and $T$ is bounded. Let $t$ be the circumcenter of $T$, and let $E'=E\times\{t\}$. Then $E'$ is a subflat of some dimension $d$ which, by ergodicity, does not depend on $(b,b')$. Now the set of subspheres of dimension $d$ of a $k_0$-dimensional Euclidean space is metric field with a $\Gamma$-invariant metric \cite[Lemma 20]{BDL}. By relative isometric ergodicity the map $(b,b')\mapsto \partial E'$ does not depend on $b'$. By the second part of \cite[Lemma 20]{BDL}, the set of Euclidean subsets of $X_b$ whose boundary is $\partial E'$ is again a metric field with an $\Gamma$-invariant metric. Thus relative isometric ergodicity again allows us to conclude that the map $(b,b')\mapsto E'$ does not depend on $b'$. In other words we get a map which associates to $b$ a subflat of $F_b$. Since we assumed the dimension of $F_b$ to be minimal, this map must be equal to $\psi_0$. This means that $f_{b,b’}$ is constant on $F_b$, and therefore $F_b$ and  $\sigma_{b’,b}(F_b^\prime)$ are parallel.

In the last case, one can write $Y=E\times T$ where $T$ is unbounded, but $\partial T$ has a center. Then we get map which associates to $(b,b')$ the center of $\partial T$, which is a point in $\partial X_b$. We conclude by the same argument as in the first case.


This concludes the proof of the claim: $F_b$ is (almost surely) parallel to $\sigma_{b’,b}(F_{b’})$. 
The set of flats parallel to $F_b$ is a metric field with an $\Gamma$-invariant metric. Therefore, one can apply again relative isometric ergodicity to prove that the map $(b,b')\mapsto \sigma_{b’,b}(F_{b’})$ does not depend on $b'$. In other words we get a map $\psi_1$ such that for almost all $b^\prime$, $\sigma_{b,b^\prime}(\psi_0(b^\prime))=\psi_1(b)$. Let us denote $G_b=\psi_1(b)$. One has, $\sigma_{b,b’}(F_{b’})=G_b$ and since $\sigma_{b’,b}=\sigma_{b,b’}^{-1}$, one has also $\sigma_{b,b’}(G_{b’})=F_{b}$. Thus, $\sigma_{b,b’}$ maps the flat equidistant to $F_{b’}$ and $G_{b’}$ to the flat equidistant to $F_b$ and $G_b$.

Up to replacing $\psi_0(b)$ by the flat equidistant to $\psi_0(b)$ and $\psi_1(b)$, we may assume that $\psi_0(b)=\psi_1(b)$ and thus $\sigma_{b,b^\prime}(F_{b^\prime})=F_b$ for almost all $(b,b^\prime)\in B\times B$. Let us recall that points in $X_b$ are positive definite subspaces $W\subset V_b^\bot/V_b$ and we denote by $\overline{W}$ the pre-image of $W$ under the quotient map $V_b^\bot\to V_b^\bot/V_b$. Let us denote $\psi(b)=\overline{F_b}\subset\mathcal{V}_k$ where $\overline{F_b}$ is the collection of $\overline{W}$ for $W\in F_b$. Since $\sigma_{b,b^\prime}(F_{b^\prime})=F_b$, if $W\in F_b$ then $W’=\sigma_{b’,b}(W)\in F_{b’}$ satisfies $\overline{W}=V_b+(\overline{W}\cap\overline{W’})$. In particular $\Span(\psi(b))\cap\Span(\psi(b’))\neq\{0\}$ for almost all $(b,b’)$. Moreover, the dimension of this intersection is essentially constant by ergodicity.

So, there is $b_0\in B$ such that for almost every $b$, $\Span(\psi(b_0))\cap\Span(\psi(b))\neq\{0\}$ and this set, $B^\prime$, of full measure can be assumed to be $\Gamma$-invariant. Let $\ell_b$ be the line corresponding to $\Span(\psi(b))$ in $\Lambda^d\mathcal{H}$ where $d$ is the dimension of $\Span(\psi(b))$. In particular, for all $b\in B'$, $\ell_b$ is in the kernel of the map

$$\begin{array}{rcl}
\Lambda^d\mathcal{H}&\to&\Lambda^{2d}\mathcal{H}\\
v&\mapsto&v\wedge v_{b_0}
\end{array}$$
where $v_{b_0}$ is a fixed non-trivial vector in $\ell_{b_0}$. As in the proof of Proposition \ref{cor}, the closure of the span of $\{\psi(b)\}_{b\in B'}$ is a non-trivial $\Gamma$-invariant subspace in $\Lambda^d\mathcal{H}$. Thanks to Lemma \ref{machin}, we have a contradiction with the Zariski-density assumption.

The statement about transversality is a direct consequence of Proposition \ref{cor}.
\end{proof}

\subsection{Low rank cases}

In this subsection, we prove that if the rank of the target is at most 2 then Zariski-density can be relaxed to geometric density to obtain the desired boundary map. The difference between the rank 1 or 2 cases and the general case comes from the complexity of the relative positions of finitely many points in $\calI_k$ for $k\leq p$. This complexity increases with $p$ but remains manageable in small ranks.
\begin{theorem}\label{thm:bnd_map_small_rank}
Let $\Gamma$ be a countable group with strong boundary $B$ and let $\rho\colon\G\to\PO_\K(p,\infty)$  be a representation. Assume that $p\leq 2$ and $\rho$  has no invariant linear subspace of dimension at most 4. Then there is a $\Gamma$-map $\phi\colon B\to\Ii_p$ such that for almost every $(b,b’)\in B\times B$, $\phi(b)$ is opposite to $\phi(b')$.
\end{theorem}

\begin{proof}Let us prove first that the induced action on $\calX_\K(p,\infty)$ is non-elementary. If there is a fixed point at infinity then there is an invariant isotropic subspace of dimension at most $p$ and if there is a flat subspace of dimension $d$, the span of its points is a subspace of dimension $2d\leq4$.

In case $p=1$, the whole visual boundary is identified with $\Ii_p$ and opposition simply means that the map is not essentially constant (which is the case, otherwise there would be an invariant isotropic line). So the existence is guaranteed by Theorem~\ref{thm:duc} and it is not constant because there is no fixed point at infinity.

Now assume $p=2$. We know the existence of a $\Gamma$-map $\phi\colon B\to\Ii_k$ with $k=1$ or $2$ by Theorem~\ref{thm:duc}. Let us prove opposition first, in both cases. If $k=1$, two isotropic lines are not opposite if they are orthogonal. By double ergodicity of  $\Gamma$, if the opposition condition is not satisfied then one can find a subset $B_1$ of $B$ of full measure such that for all $b\in B_1$ and $\gamma\in\Gamma$, $\phi(b)$ and $\phi(\gamma b)$ are orthogonal. Fix $b\in B_1$. In particular, the span of $\{\phi(\gamma b)\}_{\gamma\in\Gamma}$ is totally isotropic (thus of dimension at most 2) and $\Gamma$-invariant, contradicting the assumption.

Now if $k=2$, two distinct isotropic planes are not opposite if and only if their intersection is a line. We claim that the essential image of $\phi$ is given by isotropic planes with a common line. Let $V_b=\phi(b)$, assume the map $\phi$ is not essentially constant and choose $V_1,V_2$ distinct isotropic planes with a common line $\ell=V_1\cap V_2$ such that almost surely $V_b\cap V_i$ is a line. If $\ell$ lies in $V_b$ almost surely, then $\ell$ is $\Gamma$-invariant. So asume that $\ell$ is not essentially contained in $V_b$, then there is $V_3$ in the image of $\phi$ such that $\ell$ is not in $V_3$. So $\ell_1=V_1\cap V_3$ and $\ell_2=V_2\cap V_3$ are distinct lines and thus $V_3=\ell_1\oplus\ell_2$ lies in $V_1+V_2$. Now, for any $b’\in B$, if $V_{b’}$ contains $\ell$ then $V_{b’}$ is spanned by a $\ell$ and a line in $V_3$. If not, $V_{b’}$ mets $V_1$ and $V_2$ in two different lines. In both case $V_{b’}$ lies in $V_1+V_2$. So,  $V_b$ lies in $V_1+V_2$ which thus $\Gamma$-invariant. So we have a contradiction and thus we know that $\phi$ has the opposition property.\\

We conclude the proof by showing that if the image of $\phi$ lies in $\Ii_1$ then there is also a  $\Gamma$-map to $\Ii_2$. We rely on the beginning of the proof of Theorem~\ref{Shilov} in the Zariski-dense case before the appearance of stabilizers of subspaces in some exterior power at the very end. Let us denote $\ell_b$ for the line $\phi(b)$. In particular we can reduce to one of the following two cases: either there is an invariant section of the field  $X_b=\mathcal{X}(\ell_b^\bot/\ell_b,Q)$, or there is an invariant flat subfield not reduced to a point.

 If there is an invariant section of the field $X_b=\mathcal{X}(\ell_b^\bot/\ell_b,Q)$ then we get a map $b\mapsto V_b$ where $V_b$ is a 2-dimensional linear subspace containing $\ell_b$ and such that the signature of $Q$ on $V_b$ is $(1,0)$. We also know (by the same argument as in the proof of Theorem \ref{Shilov}) that almost surely $V_b\cap V_{b’}$ is a positive definite line which is orthogonal to $\ell_b$ and $\ell_{b’}$. If this intersection is essentially constant then we have a positive definite invariant line and we are done. 
So assume this is not the case. We can choose $V_1,V_2$ distinct in the essential image and $V_3$ that does not contain $V_1\cap V_2$. The same argument as in the proof of opposition shows that any $V_b$ in the essential image actually lies in $V_1+V_2$ and we have a contradiction, showing that it cannot happen that there is an invariant section of the field $X_b=\mathcal{X}(\ell_b^\bot/\ell_b,Q)$.

If there is an invariant flat subfield not reduced to a point in $X_b$ then it is a geodesic since $X_b$ has rank 1. This means, as before, that there is a map $b\mapsto V_b$ where $V_b$ is a 3-dimensional subspace of signature $(1,1)$ that contains $\ell_b$ (which is the kernel of the restriction of $Q$ to $V_b$). By construction of the perspectivity $\sigma_{b,b’}$, one has that almost surely $V_b\cap V_{b’}$ is a two dimensional subspace of signature $(1,1)$. If this intersection is essentially constant then we have a 2-dimensional invariant linear subspace and we are done. 

If this is not the case, then as before choose $V_1,V_2$ in the essential image of the map $b\mapsto V_b$ and $V_3$ that does not contain $V_1\cap V_2$, so $V_3\cap V_1$ and $V_3\cap V_2$ are two distinct subspaces of dimension 2. In particular, their union span $V_3$ and $V_3\leq V_1+V_2$. Now let $V_b$ be in the essential image, for the same reason as for $V_3$, either $V_b$ lies in $V_1+V_2$  or $V_b$ contains $V_1\cap V_2$, but in this last case $V_b$ meets $V_3$ in a 2-dimensional subspace that contains a line not included in $V_1\cap V_2$. So $V_b\leq (V_1\cap V_2)+V_3\leq V_1+V_2$.  Once again, we get that $V_1+V_2$ is $\Gamma$-invariant.

If there is no invariant section nor invariant flat subfield in $(X_b)$ then there is a map $\psi\colon b\mapsto\partial X_b$ which yields the desired map to $\Ii_2$.
\end{proof}
It is shown in \cite[Proposition 5.5]{MR3263898} that geometric density implies irreducibility (in the real case but the proof works over $\C$ and $\H$ as well). So we deduce straightforwardly the following.
\begin{corollary}\label{cor:bnd_map_small_rank}Let $\Gamma$ be a countable group with strong boundary $B$ and let $\rho\colon\G\to\PO_\K(p,\infty)$  be a representation. Assume that $p\leq 2$ and $\rho$  is geometrically dense then there is a $\Gamma$-map $\phi\colon B\to\Ii_p$ such that for almost every $(b,b’)\in B\times B$, $\phi(b)$ is opposite to $\phi(b')$.
\end{corollary}
\begin{remark} 
In general, irreducibility of a the representation $\G\to\PO_\K(p,\infty)$ implies non-elementarity and it is shown in \cite[Proposition 5.5]{MR3263898} that geometric density implies irreducibility. The converse of the latter implication does not hold since the embedding of $\OC(1,\infty)$ in $\OR(2,\infty)$ (given by considering the underlying real Hilbert space $\calH_\R$ and the real part of the Hermitian form) is irreducible but not geometrically dense since a copy of $\calX_\C(1,\infty)$ embeds equivariantly in $\calX_\R(2,\infty)$.

We do not expect that Theorem~\ref{thm:bnd_map_small_rank} holds for $p\geq3$ but it is likely that Corollary~\ref{cor:bnd_map_small_rank} holds for $p\geq3$.
\end{remark}
\section{Bounded cohomology and the bounded K\"ahler class}\label{sec:bndbasics}

In this section, we will recall the definitions of maximal representations, as well as adapt them in order to deal with infinite-dimensional symmetric spaces. Some familiarity with the basics on K\"ahler classes and bounded cohomology in finite dimension is advisable; the interested reader can consult for example \cite{MR2655315}.

\subsection{Bounded cohomology}
We recall here the notions from the theory of bounded cohomology that we will need in the paper. We refer the reader to \cite{MR1840942} for a thorough treatment.

The \emph{bounded cohomology} $\Hb^n(G,\R)$ of a group $G$ is the cohomology of the complex 
$$\Cb^n(G,\R)^G=\left\{f:G^{n+1}\to \R| f \text{ is  $G$-invariant, }\sup_{(g_0,\ldots,g_n)\in G^{n+1}}|f(g_0,\ldots,g_n)|<\infty\right\}$$
whose coboundary operator is defined by the formula
$$df(g_0,\ldots,g_{n+1})=\sum_{i=0}^{n+1}(-1)^i f(g_0,\ldots,\hat g_i,\ldots, g_{n+1}).$$
The bounded cohomology of discrete groups was first introduced by Gromov \cite{Gromov}, and proved to be a useful tool in proving rigidity results, in particular since it allows to detect properties of boundary maps. We will also exploit this feature in Proposition \ref{prop:chainpres} below.

Despite bounded cohomology is, in general, a much wilder theory then ordinary group cohomology (e.g. the third bounded cohomology of a free group is infinite dimensional), it admits a natural homomorphism
$$c:\Hb^n(G,\R)\to \HH^n(G,\R),$$
the \emph{comparison map}, induced by the inclusion of bounded cochains in ordinary cochains.

A second important advantage of bounded cohomology over ordinary group cohomology that will play a crucial role also in our work is that the $\ell^\infty$-norm on bounded cochain $\Cb^n(G,\R)$ induces a seminorm, the \emph{Gromov norm}, in bounded cohomology:
$$\|\kappa\|_\infty:=\inf_{[f]=\kappa}\sup_{(g_0,\ldots,g_n)\in G^n}|f(g_0,\ldots,g_n)|.$$

When $G$ is a locally compact group, Burger and Monod \cite{BM} defined the \emph{continuous bounded cohomology} $\Hcb^n(G,\R)$ of $G$  and showed that, in degree 2, the comparison map $c:\Hcb^2(G,\R)\to \Hc^2(G,\R)$ is an isomorphism when $G$ is a semisimple Lie with finite center. Here $\Hc^2(G,\R)$ denotes the continuous cohomology of $G$ (a standard text about continuous cohomology is \cite{MR1721403}). The result of Burger-Monod allows to give a complete description of $\Hcb^2(G,\R)$ in case of semisimple Lie groups with finite center: the continuous cohomology $\Hc^2(G,\R)$  can be identified with the vector space of $G$-invariant differential fom $\Omega^2(\mathcal X,\R)^G$ where $\calX$ is the symmetric space associated to $G$. In particular, for a simple Lie group  $G$  of non-compact type and finite center, 
the second continuous cohomology  $\Hcb^2(G,\R)$ is equal to $\R\kappa^{cb}_G$ if $\mathcal X$ is a  Hermitian symmetric space (and $\kappa^{cb}_G$ is then the bounded K\"ahler class, see below), and vanishes otherwise.   In general $\Hcb^2(G,\R)$ is generated by the bounded K\"ahler classes of the Hermitian factors of $\mathcal X$.

\subsection{The bounded K\"ahler class}
We now turn our attention to the bounded cohomology of the groups $G$ of isometries of the infinite dimensional Hermitian symmetric spaces $\calX$ introduced in Section \ref{sec:sym}. Since such groups $G$ are not locally compact, there is no well-established theory of continuous bounded cohomology, therefore we will just work with the bounded cohomology $\Hb^2(G,\R)$.  If $\calX$ has finite rank, we can use the K\"ahler form to define a class in the bounded cohomology of $G$, precisely as in the finite dimensional case: 
\begin{defn}
The \emph{bounded K\"ahler class} of the groups $G=\PO_\C(p,\infty)$ and $G=\PO_\R^+(2,\infty)$ is the class 
$\kappa^b_G\in \Hb^2(G,\R)$ defined by the cocycle
$$C_\omega^x(g_0,g_1,g_2)=\frac{1}{\pi}\int_{\Delta(g_0x,g_1x,g_2x)}\omega$$ 
where $x$ is any base point in the corresponding symmetric space $\calX$, $\Delta(g_0x,g_1x,g_2x)$ is the geodesic triangle with vertices $(g_0x,g_1x,g_2x)$ and $\omega$ is the Kähler form normalized such that the minimum of the holomorphic sectional curvature is -1. 
\end{defn}
The fact that $\kappa^b_G$ is independent on $x$ is proved below.
\begin{remark}\label{rk:5.2}
Let $i:H=\SU(p,q)\to G=\PO_\C(p,\infty)$ be a standard embedding. As before, we denote by $\kappa^{cb}_H\in \Hcb^2(H,\R)$  the generator corresponding, under the natural isomorphism, to the bounded K\"ahler class, and we denote by $\kappa^{b}_H\in \Hb^2(H,\R)$ the image of $\kappa^{cb}_H$ under the map induced by the inclusion of continuous bounded cochains in bounded cochains. It follows from the definition that $i^*{\kappa^b_G}=\kappa^b_H$. 
\end{remark}
\begin{lem}\label{lem:5.3}
The class $\kappa^b_G$ is well defined. Furthermore
$$\|\kappa^b_G\|_\infty=\rk(\Xx)$$
where $\Xx$ is the symmetric space associated to $G$. In particular $\kappa^b_G$ is not zero.
\end{lem}
\begin{proof}
The cocycle $C_\omega^x$ has norm bounded by $\rk(\Xx)$ since the three points $(g_0x,g_1x,g_2x)$ lie on some isometrically embedded totally geodesic copy of $\Xx_\C(p,2p)$ (resp. $\Xx_\R(2,4)$) and therefore the sharp bound of the integral computed in  \cite{DT} applies.  Furthermore the class $\kappa^b_G$ doesn't depend on the choice of the base point $x$ since for any other point $y$ the difference   $C_\omega^x-C_\omega^y$ is the coboundary of the function
$$f_\omega^{x,y}(g_0,g_1)=\frac{1}{\pi}\int_{\Delta(g_0x,g_1x,g_1y)}\omega+\frac{1}{\pi}\int_{\Delta(g_0x,g_1y,g_0y)}\omega$$
which, again, is bounded since the four points $(g_0x,g_1x,g_0y,g_1y)$ lie on some isometrically embedded totally geodesic copy of $\Xx_\C(p,3p)$ (resp. $\Xx_\R(2,6)$). Observe that, for any triple $(g_0,g_1,g_2)$ the value $C_\omega^y(g_0,g_1,g_2)-C_\omega^x(g_0,g_1,g_2)+df_\omega^{x,y}(g_0,g_1,g_2)$ is the integral of the closed form $\omega$ on a triangulation of the triangular prism with bottom face $\Delta(g_0x,g_1x,g_2x)$ and upper face $ \Delta(g_0y,g_1y,g_2y).$ This is a closed polyhedral surface contained in a finite dimensional subspace,  therefore the integral of $\omega$ over it vanishes.

In order to conclude the proof we therefore only need to show that $\|\kappa^b_G\|_\infty\geq\rk(\Xx)$. For this purpose let $\Gamma$ denote the fundamental group of a surface and let us consider the homomorphism $i:\G\to \SU(p,p)\to \PO_\C(p,\infty)$ (resp. $i:\G\to \SO^+(2,2)\to \PO_\R^+(2,\infty)$), in which the inclusion $\G\leq\SU(1,1)\to \SU(p,p)$ (resp. $\G\to \SO(2,2)$) is such that the diagonal inclusion of the Poincar\'e disk in a maximal polydisk is equivariant. It  follows from  \cite[Example 3.9]{BILW} together with Remark \ref{rk:5.2} that  $\|i^*\kappa^b_G\|_\infty=\rk(\Xx)$. Since the pullback in bounded cohomology is clearly norm non-increasing, the result follows.
\end{proof}

\subsection{The Bergmann cocycle}\label{sec:bergmann}
In the study of rigidity properties of maximal representations it will be useful to have a different representative of the bounded K\"ahler class. Such a representative will depend only on the action on a suitable boundary of the symmetric space. We distinguish two cases.

When dealing with the groups $\PO_\C(p,\infty)$, the new representative will depend on the choice of a point  $V\in\Ii_p$ the set of maximal isotropic subspaces of $\Hh(p,\infty)$. Recall that every triple $(V_0,V_1,V_2)\in (\Ii_p)^3$ is contained in a finite dimensional subspace (of dimension at most $3p$). This implies that the \emph{Bergmann cocycle}  studied in \cite{MR1923417,Clerc}\footnote{
In \cite{MR1923417,Clerc}, this cocycle is referred to as  \emph{generalized Maslov index}. We chose to denote this cocycle \emph{Bergmann cocycle}, following \cite[\S 3.2]{BIWtight}} for $\SU(p,3p)$ extends to a strict $\PO_\C(p,\infty)$-invariant cocycle
$$\beta_\C:\Ii_p^3\to [-\rk(\Xx),\rk(\Xx)]$$  
with the property that if $|\beta_{\C}(V_0,V_1,V_2)|=\rk(\Xx)$ then  $V_0,V_1,V_2$ are contained in a $2p$ dimensional subspace of signature $(p,p)$ and are pairwise transverse.

While we will not recall the explicit definition of the Bergmann cocycle (we refer to the aforementioned papers), we record its most important property:

\begin{lem}\label{lem:5.4}
For every $V\in \Ii_p$, the cocycle $C_\beta^V$ defined by
$$C_\beta^V(g_0,g_1,g_2)=\beta_{\C}(g_0V,g_1V,g_2V)$$
represents the bounded K\"ahler class. 
\end{lem}

\begin{proof}
Since any 4-tuple $(V_0,V_1,V_2,V_3)\in\Ii_p^4$ is contained in a finite dimensional subspace of $\Hh(p,\infty)$, it follows from \cite[Theorem 5.3]{Clerc} that the cocycle $C_\beta^V$ is a strict alternating bounded cocycle, cohomologous to $C_\omega^x$: the difference of the cocycles is the coboundary of a function defined similarly to the function $f_{x,y}$ in the proof of Lemma \ref{lem:5.3}, but integrating on simplices with some ideal vertices.
\end{proof}

\begin{remark}
It is worth remarking that, if $G$ is a (finite dimensional) Hermitian Lie group, the cocycles $C_\omega^x$ and $C_\beta^V$ also define a class  $\kappa^{cb}_G$ in the continuous bounded cohomology $\Hcb^2(G,\R)$. This class generates the continuous bounded cohomology $\Hcb^2(G,\R)$ for simple groups of Hermitian type.
\end{remark}

In the case of the group $\PO_\R^+(2,\infty)$, the same construction works except that the  boundary of $\calX_\R(2,n)$ on which the Bergmann cocycle is defined is $\cal{I}_1(2,n)$ and not $\cal{I}_2(2,n)$. Thus, the Bergmann cocycle for $\PO_\R^+(2,\infty)$ is a map $\beta_\R:\Ii_1(2,\infty)^3\to\{-2,0,2\}$. The fact that, in this case, the Bergmann cocycle only assumes a discrete set of possible values reflects the fact that $\OR(2,p)$ is of tube type. It is worth remarking that, in this case, the Bergmann cocycle is only preserved by the connected component of the identity in $\OR(2,\infty)$, denoted by $\OR^+(2,\infty)$.

It is possible to give an explicit description (based upon \cite[\S6]{MR2077243}) of the value of the Bergmann cocycle for triples of pairwise opposite points in $\Ii_1(2,\infty)$. For this, we need to choose  representatives $\ov x,\ov y,\ov z$  of the classes $x,y,z$ such that $Q(\ov x,\ov z)<0$ and $Q(\ov x,\ov y)<0$; furthermore, given an isotropic vector $\ov x$ in $\calH$, we denote by $[\ov x]$ the vector in $\R^2=\langle e_1,e_2\rangle$ which corresponds to the orthogonal projection (with respect to $Q$) of $[\ov x]$ and we endow $\R^2=\langle e_1,e_2\rangle$ with its canonical orientation, which allows us to determine if a triple of pairwise distinct non-zero vectors is positively or negatively oriented.  We then define (here $\rm or$ denotes the orientation):
 $$\left\{\begin{array}{lll}
 \beta_\R(x,y,z)=0 &\text{if}&Q|_{\langle x,y,z\rangle} \text{ has signature } (1,2)\\
 \beta_\R(x,y,z)=2 &\text{if}&Q|_{\langle x,y,z\rangle} \text{ has signature}  (2,1), \text{ and } \bor ([\ov x],[\ov y],[\ov z])=+\\
  \beta_\R(x,y,z)=-2 &\text{if}&Q|_{\langle x,y,z\rangle} \text{ has signature } (2,1), \text{ and } \bor ([\ov x],[\ov y],[\ov z])=-\\
 \end{array}\right.$$
 One checks that the value of $\beta_\R$ doesn't depend on the choices involved and $\beta_\R$  coincides with the Bergmann cocycle.

In order to unify the notation we will denote, from now on, by $\Ss_G$ the spaces $\Ss_{\OR(2,\infty)}:=\Ii_1(2,\infty)$ and $\Ss_{\OC(p,\infty)}:=\Ii_p(p,\infty)$. Similarly, when this will not seem to generate confusion, we will simply use the letter $\beta$ for the cocycles that we denoted before $\beta_\R$ (resp. $\beta_\C$).
\subsection{Maximal representations}
Let $\Gamma\leq \SU(1,n)$ be a lattice. 
We denote by 
$$T_b^*:\Hb^2(\Gamma,\R)\to \Hcb^2(\SU(1,n),\R)$$
the transfer map, as defined in \cite[\S 2.7.2]{MR2655315}: this is a left inverse of the restriction map $i^*:\Hcb^2(\SU(1,n),\R)\to\Hb^2(\G,\R)$ that has norm one. 
Recall that $ \Hcb^2(\SU(1,n),\R)\cong \R$ and is generated by the bounded K\"ahler class of the group $\SU(1,n)$ \cite[Lemma 6.1]{BM}.
In this section, we will denote the bounded K\"ahler class of the group  $\SU(1,n)$ by $\kappa^{cb}_n$, in order to avoid confusion with the other K\"ahler classes, and simplify the notation. 
\begin{definition}
Let $G\in\{\PO_\R(2,\infty),\PO_\C(p,\infty)\}$ and let $\rho:\G\to G$ be an homomorphism. The \emph{Toledo invariant} of the representation $\rho$ is the number $i_\rho$ such that 
\begin{equation}\label{eq:transfer}T_b^*\rho^*\kappa^b_G =i_\rho \kappa_n^{cb}\end{equation}
\end{definition}
Observe that the absolute value $|i_\rho|$ of the Toledo number  is bounded by $\rk(G)$ since both the transfer map and the pull-back are norm non-increasing. This inequality is often referred to as \emph{generalized Milnor-Wood inequality}. 
In analogy with \cite{MR2655315} we say:
\begin{definition}\label{def:max}The representation $\rho$ is \emph{maximal} if $|i_\rho|=p$.
\end{definition}

As in the finite dimensional case, it follows from the definition that the restriction of a maximal representation to a finite index subgroup is also maximal:
\begin{lemma}\label{lem:res}
The restriction of a maximal representation $\rho:\G\to G$ to a finite index subgroup $\Lambda<\Gamma$ is maximal.
\end{lemma}
\begin{proof}
Indeed denoting by $T_{b,\Lambda}^*$ (resp. $T_{b,\Gamma}^*$) the transfer map, and by $\iota^*:\Hb^2(\Gamma,\R)\to \Hb^2(\Lambda,\R)$ the isometric injection induced in bounded cohomology by the inclusion $\iota:\Lambda\to\Gamma$  (\cite[Proposition 8.6.2]{MR1840942}), one gets $T_{b,\Gamma}^*=T_{b,\Lambda}^*\iota^*$.
\end{proof}

Also the following fact descends directly from the definition, but is very useful in understanding geometric properties of maximal representations:
\begin{prop}\label{prop:nofix}
Let $\rho:\Gamma\to\PO_\C(p,\infty)$ be a maximal representation. Then there is no fixed point at infinity for $\rho$.
\end{prop}
\begin{proof}
Assume by contradiction that $\rho(\Gamma)$ fixes an isotropic subspace $V$, and choose any maximal isotropic subspace $V'$ containing $V$.
The cocycle $C_\beta^V\circ \rho$ represents the class $\rho^*\kappa_G^b$ and, since maximal triples consist of pairwise transverse subspaces, has norm  strictly smaller than $p$, thus leading to a contradiction.  
\end{proof}

We conclude this subsection observing that, as in the finite dimensional case, the pullback in bounded cohomology can be realized via boundary maps:

\begin{prop}\label{prop:chainpres}
Let $H=\PO_\C(p,\infty)$ and $\rho:\Gamma\to H$ be a maximal representation. If there exists a measurable $\rho$-equivariant boundary map $\phi:\Ii_1(1,n)\to \Ss_H$, then for every triple of pairwise distinct points $(x,y,z)\in \Ii_1(1,n)$ it holds that
$$\rk(H) \beta_{\Ii_1(1,n)}(x,y,z)=\int_{\Gamma \backslash \SU(1,n)}\beta(\phi(gx),\phi(gy),\phi(gz)){\rm d}\mu(g)$$
where $\mu$ on $\SU(1,n)/\Gamma$ is the unique $\SU(1,n)$-invariant probability measure.
\end{prop}

\begin{proof}
We use the formula in \cite[Proposition 2.38]{MR2655315}. Let $G=\SU(1,n)$, $L=G’=\Gamma$, $X=\Ss_H$ with its Borel $\sigma$-algebra. Let $\kappa’=\rho^*\kappa^b_{G}\in\Hb^2(\Gamma,\R)$  be the pullback of the bounded Kähler class and $\kappa= T_b^*\rho^*\kappa^b_{H} \in\Hcb^2(\SU(1,n),\R)$. Since $\rk(H)\beta_{\Ii_1(1,n)}$ and $\beta$ are strict alternating bounded cocycles representing respectively $\kappa$ and $\kappa'$, the formula (2.12) in \cite[Proposition 2.38]{MR2655315} yields that 
$$(x,y,z)\mapsto \rk(H)\beta_{\Ii_1(1,n)}(x,y,z)-\int_{\Gamma\backslash \SU(1,n)}\beta(\phi(gx),\phi(gy),\phi(gz)){\rm d}\mu(g)$$
is a coboundary in $L^\infty(X^3)$. By \cite[Remark 3.1]{MR2655315}, the coboundary actually vanishes and thus the equality claimed holds almost surely. Now, since both terms of the equation are everywhere defined, are $G$-invariant and satisfy the cocyle relation, the same argument as in \cite[Lemma 2.11]{Poz} proves that the equality holds for every triple of pairwise distinct points $(x,y,z)\in\Ii_1(1,n)$.\end{proof}

\subsection{Tight homomorphisms and tight embeddings}
Burger, Iozzi and Wienhard introduced, in \cite{BIWtight} the notion of tight homomorphism between Hermitian Lie groups and analogously tight embeddings between Hermitian symmetric spaces: this is of fundamental importance in the study of maximal representations since on the one hand tight homomorphisms between Lie groups can be completely classified, on the other the inclusion of the Zariski closure of the image of a maximal representation is tight; this allows, in the finite dimensional setting, to reduce the study of maximal representations to Zariski-dense maximal representations, for which construction of boundary maps is much easier.

In analogy with \cite[Definition 2.4]{BIWtight}, we define:
\begin{definition}\label{def:tight} Let $\calX$ and $\mathcal{Y}$ be (possibly infinite dimensional) Hermitian symmetric spaces of non-compact type with Kähler forms $\omega_\calX$  and $\omega_\mathcal{Y}$ associated to the Riemannian metrics of minimal holomorphic sectional curvature -1. A totally geodesic embedding $f\colon \mathcal{Y}\to\calX$ is \emph{tight} if 
$$\sup_{\Delta\subset \mathcal{Y}}\int_\Delta f^*\omega_\calX=\sup_{\Delta\subset \mathcal{X}}\int_\Delta \omega_\calX.$$

 Let $H$ be any group and $G$ be the isometry group of a (possibly infinite dimensional) Hermitian  symmetric space $\calX_G$.  Let us endow $G$ with the topology of pointwise converge, that is the coarsest topology on $G$ such $g\mapsto gx$ is continuous for any $x\in \calX_G$. Since $\calX_G$ is a complete separable metric space, it is well known that $G$ is Polish for this topology \cite[\S9.B]{MR1321597}. If  $\iota:H\to G$ is a continuous homomorphim (that is the action of $H$ on $\calX_G$ is continuous) then we denote by $\iota^*(\kappa^b_G)$ the continuous bounded cohomology class of the pull-back of the Kähler cocycle. Let us observe that this cocycle is continuous since the integration depends continuously on the vertices of the triangle. We say that $\iota$ is \emph{tight} if $$\|\iota^*(\kappa^b_G)\|_\infty=\|\kappa^b_G\|_\infty.$$
\end{definition}

Assume that, in Definition \ref{def:tight},  $H$ is the connected component of the isometry group of an irreducible Hermitian symmetric space of finite dimension and consider the homomorphism $\iota\colon H\to G$. Since geodesic triangles are contained in finite dimensional symmetric spaces, thightness of $\iota$ is equivalent to the requirement that the inclusion $\Xx_H\to \Xx_G$ of the symmetric spaces associated to $G$ and $H$ is tight \cite[Corollary 2.16]{BIWtight}:
\begin {lem}
The inclusion $\Xx_H\to \Xx_G$ of a finite dimensional  totally geodesic symmetric subspace is \emph{tight} if  and only if the inclusion $\iota:H\to G$ is tight.
\end{lem}
\begin{remark}\label{rem:hol}
If the inclusion $\mathcal X_H\to\mathcal X_G$ is totally geodesic, isometric and holomorphic, then the pullback, via the equivariant group homomorphism $\iota\colon H\to G$ of the bounded K\"ahler class, is clearly the bounded K\"ahler class. This provides many examples of tight maps: whenever the symmetric spaces have the same rank the homomorphism is tight.
\end{remark}

Let $\rho:\Gamma\to G$ be a representation of a lattice in $\SU(1,n)$ and assume that the symmetric space $\calX_G$ associated to $G$ has rank $p\in\N$. Since both pullback and transfer maps are norm non-increasing, and $\|\kappa_n^{cb}\|=1$, we deduce that $|i_\rho|\leq \|\rho^*\kappa^b_G\|_\infty$ where $i_\rho$ is defined by Equation \eqref{eq:transfer}. In particular, if the representation $\rho$ is maximal, then  $\|\rho^*\kappa^b_G\|_\infty=p$. The same argument gives:
\begin{lemma}\label{lem:tight}
Assume that a maximal representation $\rho:\G\to G$ preserves a totally geodesic Hermitian symmetric subspace $\Yy\subset \Xx_G$. Then the inclusion $\Yy\to \Xx_G$ is tight.
\end{lemma}

\subsection{Reduction to geometrically dense maximal representations}

In this subsection, we explain how to reduce the understanding of maximal representations to  geometrically dense maximal representations.

Recall that the totally geodesic subspaces of $\Xx_\C(p,\infty)$ are products of irreducible factors that are either finite dimensional or are isomorphic to either $\Xx_\C(q,\infty)$ or $\Xx_\R(q,\infty)$ or $\Xx_\H(q,\infty)$ \cite[Corollary 1.9]{Duc15}. 

\begin{prop}\label{prop:2}
Let $\rho:\Gamma\to\PO_\C(p,\infty)$ be a maximal representation. There is a minimal $\Gamma$-invariant totally geodesic subspace $\calY$ of $\Xx_\C(p,\infty)$. This space $\calY$ splits isometrically as a direct product $\calY=\calY_1\times\ldots\times\calY_k$ (possibly reduced to a unique factor), and for each $i$, either $\calY_i$ is finite dimensional Hermitian, or  it is isometric to $\Xx_\C(m,\infty)$ and the restricted representation $\rho_i:\G\to\Isom(\mathcal Y_i)$ is maximal and geometrically dense. 
\end{prop}

\begin{proof}
Since the representation is maximal, there is no fixed point at infinity (Proposition \ref{prop:nofix}). Thus there is a minimal totally geodesic $\Gamma$-invariant subspace $\calY$ (otherwise ~\cite[Proposition 4.4]{Duc13} would yield a fixed point at infinity). Since $\calY$ is a totally geodesic subspace of $\Xx_\C(p,\infty)$, it is a symmetric space of non-positive curvature operator and finite rank. Thus $\calY$ decomposes as a product $\calY=\calY_0\times\calY_1\times\cdots\times \calY_k$ where  each $\calY_k$ is a symmetric space of finite dimension of non-compact type, the Euclidean de Rham factor  or the symmetric space associated to some $\OO(l,\infty)$ with $\K=\R$, $\C$ or $\H$ (\cite[Corollary 1.10]{Duc15}). Up to passing to a finite index subgroup, we may assume that $\Gamma$ preserves each factor of this splitting (see Lemma \ref{lem:res}). Since $\calY$ is minimal as $\Gamma$-invariant totally geodesic subspace, the induced action $\Gamma\action\calY_i$ is minimal as well.

Recall that a geodesic segment in $\calY$ has  the form $\sigma(t)=(\sigma_1(t),\dots,\sigma_k(t))$ where each $\sigma_i$ is a curve with constant speed (which may vary from factor to factor). 
Since the inclusion $\calY\subset \Xx$ is tight (Lemma \ref{lem:tight}) we have
$$\sup_{\Delta\subset \calY}\int_\Delta \omega_\Xx=\sup_{\Delta\subset \Xx}\int_\Delta \omega_\Xx$$
where $\Delta$ is a geodesic triangle. Let $\Delta_i$ be the projection of  $\Delta$ to the factor $\calY_i$. The triangle $\Delta_i$ is completely determined by three points. If $\calY_i$ has infinite dimension, these three points are given by three  positive definite linear subspaces and thus are included in some standard embedding of $\Xx_\K(l,2l)$ in $\calY_i$. We denote by $\calY_i^0$ either $\calY_i$, if the subspace already has finite dimension or the image of a standard embedding of $\Xx_\K(l,2l)$ in $\calY_i$. Finally we denote by $\calY^0$ the product $\calY_0^0\times\cdots\times \calY_k^0$. The symmetric space $\calY^0$ has finite dimension,  and, since the isometry group of $\Xx_\K(p,\infty)$ acts transitively on standard embeddings of $\Xx_\C(l,2l)$, we have

$$\sup_{\Delta\subset \calY^0}\int_\Delta \omega=\sup_{\Delta\subset \calY}\int_\Delta \omega=\sup_{\Delta\subset \Xx}\int_\Delta \omega.$$

Now $\calY^0$ lies in some standard copy $\Xx^0$ of $\Xx_\C(p,N)$ for $N\geq p$ (Lemma \ref{lem:fdtgs}) and thus the embedding of $\calY^0$ in $\Xx^0$ is tight and we can apply \cite[Theorem 7.1]{BIWtight}. So, we know that each $\calY_i^0$ is Hermitian of non-compact type and the Euclidean de Rham factor is trivial. In particular all $\calY_i$ of finite dimension are Hermitian  and the infinite dimensional ones are a priori $\Xx_\C(m,\infty)$ and $\Xx_\R(2,\infty)$ but the latter is impossible. To see that, 
assume by contradiction that there is factor $\calX_\R(2,n)$ in $\calY^0$. Then the classification of tight embeddings obtained in \cite[\S5.2]{HamPoz} implies that the image of the tube type factors via the embedding lies in some totally geodesic copy of $\calX_\C(m,m)$. By considering the rank, we have $m\leq p$ and by considering the dimension, this gives a finite bound on $n$. Thus there is no factor $\Xx_\R(2,\infty)$ in $\calY$.\\
Finally, the fact that each representation $\rho_i$ is maximal is standard (see for example \cite[Lemma 2.6 (4)]{BIWtight}).
\end{proof}

\section{Representations  in $\PO_\C(p,\infty)$}\label{sec:hyplat}

We first focus on the case where the target is $\PO_\C(p,\infty)$, in this case we prove that the representation preserves a finite dimensional totally geodesic subspace.
Thus, in particular, if the domain $\G$ is a lattice in $\SU(1,n)$ for $n\geq2$, superrigidity of maximal representations $\rho:\G\to G$ where $G$ is a Hermitian Lie group applies  \cite{BI, Poz,KM}\footnote{The only case that is still open is the case of  representations of non-uniform lattices in groups of tube type.}. To this aim 
we need a good understanding of the geometry of the boundaries $\Ii_1(1,n)$ and $\Ii_p(p,\infty)$.
\subsection{The geometry of the boundary of $\HCn$}\label{sec:4.3}

Recall that a chain $\calC\subseteq \partial\HCn$ is the boundary of a totally geodesic holomorphic disk $\mathcal D\subseteq \HCn$, and is the intersection of $\partial\HCn\subseteq\C\P^n$ with a complex projective line in $\C\P^n$. For this reason, a chain is uniquely determined by two points belonging to it. Given two distinct points $x,y\in\partial\HCn$ we will denote by $\calC_{x,y}$ the unique chain containing the points $x$ and $y$.
More generally, for every $k$-dimensional subspace $\P V\subseteq \C\P^n$ that intersects $\HCn$, the subspace
$\P V$ intersects $\HCn$ in a totally geodesic submanifold isometric to $\HCk$ and intersects $\partial\HCn$ in a $2k-1$-dimensional sphere $\partial \HCk$. Following \cite{Goldman}, we will call any such sphere a $k$-\emph{hyperchain}.

Part of the work of \cite{Poz} was aimed at showing that a similar picture exists in higher rank: any two transverse subspaces $X,Y\in\Ii_p(p,\infty)$ determine a $2p$-dimensional subspace $\langle X,Y\rangle$ and therefore a finite dimensional totally geodesic subspace $\Xx_\C(p,p)\subset\Xx_\C(p,\infty)$ as well as a subset $\Ii_p(p,p)\subset\Ii_p(p,\infty)$. As in \cite{Poz}, we will refer to these subsets as \emph{$p$-chains} or merely \emph{chains}.

Here and in the rest of the article, when we will deal with differentiable manifolds, \emph{almost surely} will mean for a set of full measure in the  Lebesgue measure class.  We say that a measurable map  $\phi:\partial\HCn\to\Ii_p(p,\infty)$ \emph{almost surely maps chains to chains} if for almost every chain $\calC\subseteq \partial\HCn$ there is a $p$-chain $\calT\subset \Ii_p(p,\infty)$ such that for almost every point $x\in\calC$, $\phi(x)\in\calT$. In this case, we say that the chain $\calC$ is \emph{generic for $\phi$}.
 In order to guarantee that a measurable map $\phi$ almost surely maps chains to chains, it is enough to check that for almost every pair $(x,y)\in\partial\HCn\times\partial\HCn$, the subspaces $\phi(x)$ and $\phi(y)$ are transverse and that for almost every $z\in \calC_{x,y}$, the subspace $\phi(z)$ is contained in $\langle\phi(x),\phi(y)\rangle$ (this statement can be found e.g. in \cite[Lemma 4.2]{Poz}). In this case, we say that the pair $(x,y)$ is \emph{generic for $\phi$}. 

 A consequence of Proposition  \ref{prop:chainpres} is the following:
\begin{cor}\label{chainstochains} Let $\G<\SU(1,n)$ be a lattice. Assume that a representation $\rho:\G\to\PO_\C(p,\infty)$ is maximal and admits an equivariant boundary map $\phi:\partial\Xx_\C(1,n)\to\Ii_p(p,\infty)$. Then the  boundary map $\phi$ almost surely maps chains to chains.
\end{cor}
\begin{proof}
Observe that, since $\SU(1,1)$ acts transitively on positively oriented triples in $\Ii_1(1,1)$, there are precisely two $\SU(1,n)$ orbits of pairwise distinct triples of points on a chain. Since the equality in Proposition \ref{prop:chainpres} holds for every triple $(x,y,z)$, we deduce that for almost every triple $(x,y,z)$ on a chain, the triple $(\phi(x),\phi(y),\phi(z))$ is contained in a $p$-chain and consists of transverse points. Hence $\phi$ almost surely maps chains to chains.
\end{proof}
The purpose of the rest of the section will then be to show the following proposition:
\begin{prop}\label{prop:1}
If $\phi:\partial\HCn\to\Ii_p(p,\infty)$ is measurable and almost surely maps chains to chains, then there exists a finite dimensional, totally geodesic subspace $\Xx_{p,np}\subset\Xx_{p,\infty}$
 such that $\phi(\partial\HCn)\subset \partial \calX_{p,np}$ up to discarding a null subset of $\partial\HCn$.
\end{prop}
The proof of Proposition \ref{prop:1} is a measurable version of an easy geometric construction (Lemma \ref{lem:1}). Compare to \cite{BI} and \cite{Poz} for similar statements and arguments. 
In order to prove the proposition, we will need several easy lemmas, the first of which is a straightforward consequence of Fubini's theorem:
\begin{lemma}\label{lem:Fubini}
Let $A,B$ be differentiable manifolds, and $\pi:A\to B$ be a smooth fibration. Then
\begin{enumerate}
\item if $\mathcal O\subseteq B$ has full measure, then $\pi^{-1}(\mathcal O)\subset A$ has full measure; 
\item if $\mathcal Y\subseteq A$ has full measure, then for almost every $x\in B$, $\mathcal Y\cap\pi^{-1}(x)$ has full measure in $\pi^{-1}(x)$. 
\end{enumerate}  
\end{lemma}
In the proof of Proposition \ref{prop:1}, we will argue by induction on the dimension $n$ of $\partial\HCn$. In particular, in order to have at our disposal the inductive step, we will need the following:
\begin{lem}\label{lem:a}
If the measurable map $\phi:\partial\HCn\to\Ii_p(p,\infty)$ almost surely maps chains to $p$-chains, then for every $1\leq k\leq n$, and for almost every $k$-hyperchain $\partial\HCk\subset\partial\HCn$, the restriction $\phi|_{\partial\HCk}:\partial\HCk\to\Ii_p(p,\infty)$ almost surely maps chains to chains.
\end{lem}
\begin{proof}
This is an application of Lemma \ref{lem:Fubini}: consider the configuration spaces
$$\begin{array}{l}
\mathcal E_1^k=\{(C,X)|\; C\text{ is a chain}, X\text{ is a $k$-hyperchain}, C\subseteq X\},\\
\mathcal E_0=\{C|\;C\text{ is a chain}\}.\end{array} $$
Clearly, for every $k$, there is a smooth surjection $\mathcal E_1^k\to\mathcal E_0$. In particular the subset $\Yy\subset \mathcal E_1^k$ consisting of pairs $(C,X)$ such that $C$ is generic for $\phi$ has full measure. Lemma \ref{lem:Fubini} (2) implies the desired statement.
\end{proof}

In the inductive step, we will need to increase the dimension of $\partial\HCn$ by one. For this purpose the following additional configuration spaces will be handy:
$$\begin{array}{l}
\calF_0=\{(C,X)|\; C \text{ is a chain, } X\text{ is a $(n-1)$-hyperchain, } C\cap X\; \text{is a point}\}\\
\calF_1=\{(C,X,c,x)|\; (C,X)\in\calF_0,\; c\in C,\; x\in X \}\\
\calF_2=\{(C,X,c)|\;(C,X)\in\calF_0, c\in C\}\\
\calF_3=\{(C,X,x)|\;(C,X)\in\calF_0,  x\in X\}
\end{array}$$

\begin{lem}\label{lem:0}
Assume $\phi:\partial\HCn\to\Ii_p(p,\infty)$ almost surely maps chains to chains. Then for almost every pair $(C,X)\in \calF_0$, it holds:
\begin{enumerate}
\item the chain $C$ is generic for $\phi$;
\item for almost every pair $(c,x)\in C\times X$, the pair $(c,x)$ is generic for $\phi$;
\item for almost every point $c\in C$, the pair $(c, C\cap X)$ is generic for $\phi$;
\item for almost every point $x\in X$, the pair $(x, C\cap X)$ is generic for $\phi$.
\end{enumerate}
\end{lem}
\begin{proof}
Since the intersection of finitely many full measure subsets has full measure, it is enough to verify that each condition holds on a full measure set.
The pairs for which condition (1) holds have clearly full measure: by assumption almost every chain is generic, and $\calF_0$ smoothly fibers over the set of all chains.
 
In order to verify (2) observe that, since $\phi$ almost surely maps chains to chains, the set of pairs $(c,x)\in\partial\HCn\times\partial\HCn$ that are generic for $\phi$ has full measure. Consider now the forgetful map $\pi:\calF_1\to \partial\HCn\times\partial\HCn$. If we restrict to the open dense subset of $\calF_1$ consisting of 4-tuples  $(C,X,c,x)$ such that $c, x$ and $C\cap X$ are pairwise distinct, $\pi$ gives a surjective fibration onto the (open and dense) set of transverse pairs in  $\partial\HCn\times\partial\HCn$. In particular we deduce from Lemma \ref{lem:Fubini} (1) that the set of 4-tuples $(C,X,c,x)\in \calF_1$ such that $(c,x)$ is generic for $\phi$ has full measure in $\calF_1$. Since $\calF_1\to\calF_0$ is a smooth fibration, the statement is then a direct consequence of Lemma \ref{lem:Fubini} (2).

In order to verify that the last two conditions hold on a full measure set  as well, we use a similar argument for the fibrations $\calF_2\to \partial\HCn\times\partial\HCn$ and  $\calF_3\to \partial\HCn\times\partial\HCn$ given respectively by $(C,X,c)\mapsto (C\cap X,c)$ and 
$(C,X,x)\mapsto (C\cap X,x)$.
\end{proof}
The inductive step will be based on the following construction:
\begin{lem}\label{lem:1}
For any pair $(C,X)$ in $\calF_0$, the union
$$S=\bigcup_{\substack{c\in C\\x\in X}}\calC_{c,x}$$ contains an open and dense subset of $\partial\HCn$. 
\end{lem}

\begin{center}
\begin{figure}
\begin{tikzpicture}[scale=.8]
\draw(1,3) to (1,9) to (3,10) to (3,6);
\draw (1,3) to (3,4) to (3,5);
\node at (2.5,9){$X$};
\node at (.5,4.7){$S_X$};
\draw (1,5) to (3,6) to (10,6) to (8,5) to (1,5); 
\draw (1,5) to (0,5) to (1, 5.5);
\node at (8,5.5){$\C^{n-1}$};
\draw (6, 10) to (6, 5.5);
\draw(6, 5) to (6,3.5);
\node at (6,9.5)[right]{$C$};
\node at(6,5.5)[right]{$p_C$};
\filldraw (6,5.5) circle[radius=1pt];
\filldraw (4,8.75) circle [radius=1pt];
\node at (4,8.75)[above]{$y$};
\draw (4,8.5) ellipse (2cm and .25cm);
\filldraw (2,8.5) circle [radius=1pt];
\node at (2,8.5) [left] {$x$};
\filldraw (6,8.5) circle [radius=1pt];
\node at (6,8.5) [right] {$c$};
\node at (4,8.75)[below]{$\mathcal C$};

\draw (4,5.5) ellipse (2cm and .25cm);
\filldraw (2,5.5) circle [radius=1pt];
\node at (2,5.5) [above] {$z_y$};
\filldraw (4, 5.75) circle[radius=1pt];
\node at (4, 5.75) [below] {$\pi(y)$};
\end{tikzpicture}
\caption{The proof of Lemma \ref{lem:1}}
\end{figure}
\end{center}
\begin{proof}
We work in the Heisenberg model for $\partial\HCn$ in which the intersection point $C\cap X$ corresponds to $\infty$. It is well known that in this model, isomorphic to $\C^{n-1}\ltimes \R$, a chain $W$ corresponds to either a vertical line or to a topological circle that  projects to an Euclidean circle $E$ contained in an affine complex line  $L\subset \C^{n-1}$. Moreover, denoting by $\pi:\C^{n-1}\ltimes \R\to\C^{n-1}$ the projection, for every Euclidean circle $E\subset\C^{n-1}$, and every point $x\in\pi^{-1}(E)$, there exists a unique chain $W$ containing $x$ and satisfying $\pi(W)=E$ \cite[Section 4.3]{Goldman}.

Since we chose the Heisenberg model in which $C\cap X$ corresponds to $\infty$, the chain $C$ corresponds to a vertical line (pre-image of the point $p_C\in\C^{n-1}$), and the $(n-1)$-hyperchain $X$ corresponds to the pre-image under $\pi$ of a $n-2$ dimensional affine subspace $S_X$ of $\C^{n-1}$. If $\langle S_X,p_C\rangle_\R$ denotes the  $\R$-affine span of the two affine subspaces of $\C^{n-1}$, we will prove that the open dense subset $$\pi^{-1}\left(\C^{n-1}\setminus \langle S_X,p_C\rangle_\R\right)$$ is contained in $S$.

Indeed, for any point $y$ in  $\C^{n-1}\ltimes \R$ such that $\pi(y)$ doesn't belong to $\langle S_X,p_C\rangle_\R$, the complex affine line determined by $\pi(y)$ and $p_C$ intersects $S_X$ in a unique point $z_y$. The three points $(p_C, z_y, \pi(y))$ are not $\R$-colinear, and determine a unique Euclidean circle $E_y$. The unique chain $\mathcal C$ projecting to $E_y$ and containing $y$ will, by construction, intersect $C$ in a point $c$ and $X$ in a point $x$, which shows that $y\in S$.
\end{proof}

We can now conclude the proof of Proposition \ref{prop:1}:
\begin{proof}[Proof of Proposition \ref{prop:1}]
We argue by induction.

In case $n=1$, we know that for almost every positively oriented triple $(x,y,z)$, the triple $(\phi(x),\phi(y),\phi(z))$ is contained in a $2p$-dimensional linear subspace of the Hilbert space  $\Hh$ (as in \S 2.2) of signature $(p,p)$. By Fubini’s theorem, one can find $x,y$ such that $\phi(x)$ and $\phi(y)$ are opposite (i.e. span a subspace of signature $(p,p)$) and for almost all $z$, $\phi(z)$ lies in the span  $\langle\phi(x),\phi(y)\rangle$. 

For the inductive step, combining Lemma \ref{lem:a} and Lemma \ref{lem:0}, we deduce that the set $\mathcal A$ of pairs $(C,X)\in\calF_0$ such that the restriction of $\phi$ to $X$ almost surely maps chains to chains, and such that all conditions of Lemma \ref{lem:0} hold for $(C,X)$, has full measure in $\calF_0$. 
In particular $\mathcal A$ is not empty, and we can chose a pair $(C,X)\in \mathcal{A}$. 

By the inductive hypothesis and Lemma \ref{lem:0} (4), there is a $np$-dimensional linear subspace $V_{p,(n-1)p}$ of $\mathcal H$ such that $\phi(C\cap X)< V_{p,(n-1)p}$ and for almost every $x\in X$, $\phi(x)< V_{p,(n-1)p}$. Let us choose a point $y$ in $C$ such that the pair $(y, C\cap X)$ is generic for $\phi$ and define
$$V_{p,np}=\langle V_{p,(n-1)p},\phi(y)\rangle.$$ 

Since the pair $(y, C\cap X)$ is generic for $\phi$, for almost every point $c\in C$, $\phi(c)<  V_{p,np}$. Since, by Lemma \ref{lem:0} (2), almost every pair $(c,x)\in C\times X$ is generic, there exist a full measure subset of $S=\bigcup \mathcal C_{c,x}$ consisting of points $s$ with $\phi(s)<  V_{p,np}$. The conclusion follows since, by Lemma \ref{lem:1}, the set $S$ contains an open dense subset of $\partial \HCn$ and hence a full measure subset of $S$ has full measure in $\partial \HCn$.  
\end{proof}

\subsection{Rigidity of maximal representations of complex hyperbolic lattices}
We now have all the needed ingredients to prove our rigidity result for maximal representations of complex hyperbolic lattices.
\begin{thm}\label{thm:hyplatbndmap} Let $n\geq 2$ and let $\G<\SU(1,n)$ be a complex hyperbolic lattice, and let  $\rho:\Gamma\to\PO_\C(p,\infty)$ be a maximal representation. If there is a $\rho$-equivariant measurable map $\phi\colon \partial\HCn\to \calI_p$ then there is a finite dimensional totally geodesic Hermitian symmetric subspace $\calY\subset\Xx_\C(p,\infty)$ that is invariant by $\Gamma$. Furthermore, the representation $\Gamma\to\Isom(\calY)$ is maximal.
\end{thm}

\begin{proof}

By Corollary~\ref{chainstochains} and Proposition \ref{prop:1}, we know that the image of $\phi$ is essentially contained in the boundary of some $\Xx_\C(p,np)$. Since $\Gamma$ is countable, we can find a $\Gamma$-invariant full measure subset of $\partial\HCn$ whose image in contained in $\partial\Xx_\C(p,np)$. In particular, this copy of $\Xx_\C(p,np)$ is $\Gamma$-invariant. This concludes the proof.
\end{proof}

\begin{proof}[Proof of Theorem \ref{thm:hyplat}]Under the hypothesis of Zariski-density, a measurable $\rho$-equivariant map $\phi: \partial\HCn\to \Ii_p$ is given by Theorem \ref{Shilov}. If $p\leq 2$, we know from Proposition \ref{prop:2} that the representation $\rho$ virtually splits as a product of geometrically dense maximal representations, and therefore is enough to understand the case in which $\rho$ is geometrically dense. In this case, the existence of a measurable $\rho$-equivariant map $\phi\colon\partial\HCn\to \Ii_p$ is given by Corollary~\ref{cor:bnd_map_small_rank}.
\end{proof}
\begin{remark}
Combining the results of this paper and those of \cite{KM} one can deduce that if $\Gamma<\SU(1,n)$ is cocompact, and  $\rho:\Gamma\to \OC(p,\infty)$ is maximal then there is a  totally geodesic subspace $\Yy\subset \Xx_\C(p,\infty)$ preserved by $\rho(\Gamma)$ which is isometric to $\Xx_\C(1,n)\times \Xx_\C(p_1,\infty)\times\ldots\times \Xx_\C(p_k,\infty)$ after a suitable rescaling of the metric of the various factors, where $p_i>2$. Furthermore the induced action on $\Xx_\C(p_i,\infty)$ is maximal and geometrically dense, but not Zariski-dense. If $\Gamma<\SU(1,n)$ is non-uniform, we can deduce from \cite{Poz} the same result where possibly $\Yy$ also has some finite dimensional factors of tube type. We conjecture that indeed $\Yy=\Xx_\C(1,n)$: the absence of factors of type $\Xx_\C(p_i,\infty)$ would be implied by a positive answer to Question \ref{qu:1.5}.
\end{remark}



\section{Maximal representations in $\PO_\R(2,\infty)$}\label{sec:bndcohomo}

In this section, 
we will construct the example of Theorem \ref{thm:dense}. In this section we focus on lattices $\G_\Sigma<{\rm PSL}(2,\R)=\PU(1,1)$. A finite index subgroup of $\G_\Sigma$ is then the fundamental group of a Riemann surface of negative Euler characteristic.


In order to construct geometrically dense maximal representations, we need to recall a bit of the geometry of $\calX_\R(2,\infty)$ and of the specific  boundary where the Bergmann cocycle is defined. 
 Recall from Section \ref{sec:bergmann} that 
the Bergmann cocycle
$\beta_\R:\Ii_1(2,\infty)^3\to \{-2,0,2\}$ induces a $\OOO_\R^+(2,\infty)$-invariant partial cyclic ordering on the set of isotropic lines $\Ii_1(2,\infty)$: we say that $(x,y,z)$ is \emph{positively oriented} if and only if $\beta_\R(x,y,z)=2$. This is a consequence of the fact that $\beta_\R$ is a cocycle, and hence if $\beta_\R(x,y,z)=2$ and $\beta_\R(x,z,t)=2$,  then  necessarily $\beta_\R(x,y,t)=2$ and $\beta_\R(y,z,t)=2$. We say that a triple $(x,y,z)\in\Ii_1(2,\infty)$ is \emph{maximal} if  $\beta_\R(x,y,z)=2$. It is easy to check that maximal triples form a single  $\OOO^+_\R(2,\infty)$-orbit. More generally, we say that an $n$-tuple $(x_1,\ldots,x_n)$ is \emph{maximal} if every  subtriple $(x_i,x_j,x_k)$, with $i\leq j\leq k$, is. Furthermore, given an opposite pair $(x,z)\in\Ii_1(2,\infty)$ we denote by $I_{x,z}$ the \emph{interval} with endpoints $(x,z)$:
$$I_{x,z}=\{y\in \Ii_1(2,\infty)|\;(x,y,z)\text{ is maximal}\}.$$
The following property of intervals will be useful:
\begin{prop}\label{prop:Iconv}
Let $x,y$ be a pair of opposite points in $\cal{I}_1(2,\infty)$. The interval $I_{x,y}$ is homeomorphic to a bounded convex subspace of a Hilbert space. 
\end{prop}
\begin{proof}
Recall from Section \ref{sec:bergmann} that given two opposite isotropic lines $x,y\in\Ii_1(2,\infty)$ of which we choose representatives $\ov x,\ov y$ such that $Q(\ov x,\ov y)<0$, the interval $I_{x,y}$ consists of the isotropic subspaces
$$I_{x,y}=\{z\in\Ii_1(2,\infty)|\;Q(\ov x,\ov z)<0, Q(\ov y,\ov z)<0, \text{ and } \bor ([\ov x],[\ov y],[\ov z])=+\}$$
indeed the expression of the restriction of the quadratic form to the subspace $x,y,z$ is represented, with respect to that basis $\{\ov x, \ov y, \ov z\}$ by the matrix
$$\left(\begin{smallmatrix} 0& Q(\ov x,\ov y)& Q(\ov x,\ov z)\\Q(\ov x,\ov y)&0&Q(\ov y,\ov z)\\Q(\ov x,\ov z)&Q(\ov y,\ov z)&0\end{smallmatrix}\right)$$
whose determinant, $2Q(\ov x,\ov y)Q(\ov x,\ov z)Q(\ov y,\ov z)$ is negative if and only if the signs of $Q(\ov x,\ov z)$ and $Q(\ov y,\ov z)$ are equal and can be chosen negative.

Without lost of generality, we can find a Hilbert basis $(e_i)_{i\in\N}$ which is orthogonal for $Q$, such that $Q(e_1)=Q(e_2)=1$, $Q(e_i)=-1$ for $i\geq 3$ and such that $x,y$ have representatives $\overline{x}=e_1+e_3$ and $\overline{y}=-e_1+e_3$. For $z\in\mathcal{I}_1(2,\infty)$, let $\overline{z}$ be a representative of $z$ such that $||\overline{z}||=\sqrt{2}$ (here the norm $||\cdot||$ is computed with respect to the scalar product $\langle\ ,\ \rangle$ defining the Hilbert space $\mathcal H$). We can  write  $\overline{z}=u+v$ with $||u||=||v||=1$, $u$ in the span of $\{e_1,e_2\}$ and $v$ in the orthogonal of $\{e_1,e_2\}$. If we write $u=u_1e_1+u_2e_2$ and $v=v_3e_3+v’$ with $v’\bot e_3$, then the requirements $Q(\ov x,\ov z)<0$ and $Q(\ov y,\ov z )<0$ are both satisfied if and only if $v_3>|u_1|$, furthermore in this case $\bor([\ov x],[\ov z],[\ov y])=+$ if and only if $u_2>0$. In particular $\Ii_1(2,\infty)$ is homeomorphic to the pairs $(u_1,v')$ with $|v'|^2+|u_1|^2<1$ which is a bounded convex subset of a Hilbert space.   
 \end{proof}

In analogy with the finite dimensional case, we say that an element $g\in\OR^+(2,\infty)$ is \emph{Shilov-hyperbolic} if it has an attractive line in $\Ii_1(2,\infty)$, or equivalently if  $g$ has a real eigenvalue $\lambda_1(g)$ of absolute value strictly bigger than one and multiplicity one (observe that $g$ has at most two eigenvalues of absolute value bigger than one, and in this case we denote by $\lambda_1(g)$ the eigenvalue with highest absolute value). If $g$ is Shilov-hyperbolic, we denote by $g^+\in\Ii_1(2,\infty)$ the eigenline corresponding to $\lambda_1(g)$ and by $g^-\in\Ii_1(2,\infty)$ the eigenline corresponding to $\lambda_1(g)^{-1}$.

In order to carry out our construction of geometrically dense maximal representation, we will need the following result, which ensures a good nesting property of the images of intervals under the action of Shilov hyperbolic elements:
\begin{prop}\label{prop:intervals}

Given a Shilov-hyperbolic element $g$ and a maximal 4-tuple $(x,y,z,t)\in\Ii_1(2,\infty)$ such that also $(x,y,g^+,z,t,g^-)$ is maximal, there exists $n\in\N$ such that $(y,g^nx,g^+,g^nt,z)$ is maximal as well.

\end{prop}
\begin{proof}
As in the proof of Proposition \ref{prop:Iconv}, we fix a Hilbert basis of $\mathcal H$ such that $Q(e_1,e_1)=Q(e_2,e_2)=1$, and $Q(e_i,e_i)=-1$ for all $i\geq 3$.
Since the group $\OR(2,\infty)$ acts transitively on pairs of opposite isotropic lines  we can, without loss of generality, assume that $\ov g^+=e_1+e_3$, and $\ov g^-=-e_1+e_3$. Since the 6-tuple $(x,y,g^+,z,t,g^-)$ is maximal, we can fix, for every $w\in\{x,y,z,t\}$ a lift $\ov w$ of the form $(\cos \theta_w,\sin\theta_w, w_3,\ldots)$ with the additional requirements that $\sum_{i\geq 3} w_i^2=1$ (because $w$ defines an isotropic line), and that $w_3>0$. Furthermore, since the restriction of $Q$ to $\langle w,g^+,g^-\rangle$ has signature $(2,1)$, we deduce that $Q(\ov x, \ov g^+)$ and $Q(\ov x,\ov g^-)$ have the same sign, and in particular $w_3>|\cos\theta_w|$. Finally, the maximality of the 6-tuple $(x,y,g^+,z,t,g^-)$ implies that $$-\pi <\theta_x<\theta_y<0<\theta_z<\theta_t<\pi.$$ 

We  decompose each vector $\ov w$ in the relevant eigenspaces for $g$: $\ov{w}=w^+\ov{g}^++w^-\ov{g}^-+w_0$ (where the vector $w_0$ is orthogonal to $\langle g^+,g^-\rangle$). Observe that $w^+=(w_3+\cos\theta_w)/2$ and $w^-=(w_3-\cos\theta_w)/2$. Since we know that  $w_3>|\cos\theta_w|$, we deduce that $w^+\neq 0$ and since $\|g^n(w^+\ov{g}^+)\|/\|g^n(\ov{w}-w^+\ov{g}^+)\|\geq\left|\lambda_1(g)/\lambda_2(g)\right|^n$, where $\lambda_2(g)$ is the second maximal eigenvalue (possibly of absolute value 1), we can find $n$ big enough such that  $\theta_y<\theta_{g^nx}<0<\theta_{g^nt}<\theta_z$. 
Up to consider $g^2$ instead of $g$, we may assume that $\lambda_1(g)>0$ and since $w_3>0$, we have $g^n\ov w/\|g^n\ov w\|\to\ov{g}^+/\|\ov{g}^+\|$. So by continuity of $Q$, $Q(g^n\ov{w},y)$ has the same sign as $Q(\ov{g}^+,y)$ for $n$ large enough. 

Hence we can find $n$ such that the restriction of $Q$ to $\langle y,g^nx,g^+\rangle$ has signature $(2,1)$ and  $\theta_y<\theta_{g^nx}<0$ which implies that the orientation $\bor ([\ov y],[g^n \ov x], [\ov{g}^+])$ is positive. For such $n$, the triple $(y,g^nx,g^+)$ is maximal and similarly, we can also assume, up to possibly further enlarging $n$, that $(g^+,g^nt,z)$ is maximal for $n$ large enough.
Together with the fact that $(g^nx,g^+,g^nt)$ is maximal for any $n$, this is enough to guarantee that $(y,g^nx,g^+,g^nt,z)$ is maximal for $n$ large enough.
\end{proof}

\begin{prop}\label{prop:dense}
There exists a maximal 4-tuple $(x,y,z,t)\in\Ii_1(2,\infty)$ and  Shilov-hyperbolic elements $g,h\in\OOO_\R^+(2,\infty)$ such that the 8-tuple $(x,h^+,y,g^+,z,h^-,t,g^-)$ is maximal and such that the group generated by $g$ and $h$ doesn't preserve any finite dimensional subspace of $\calH$.
\end{prop}
\begin{proof}
We decompose the Hilbert space $\calH$ as a direct sum $\calH=\cal V\oplus \cal W$ where $\cal V$ and $\cal W$ are orthogonal with respect to $Q$, the restriction of $Q$ to $\cal V$ has signature $(2,2)$ and $(\calW,-Q|_\calW)$ is a Hilbert space with Hilbert basis $(e_i)_{i\geq 3}$. We choose an element $g\in\OOO_\R^+(2,\infty)$ that induces a Shilov hyperbolic element of $\cal V$, and acts on each subspace $\cal L_i:=\langle e_{2i+1},e_{2i+2}\rangle$ of $\calW$ as a rotation of angle $\theta_i$ where $\theta_i/\pi$ are distinct irrational numbers modulo 2. Observe that the attractive (resp. repulsive) eigenlines $g^\pm$ of $g$ belong to $\Ii_1(\cal V)\subset\Ii_1(2,\infty)$. Furthermore, every invariant subspace for the $g$ action is obtained as the direct sum of a subspace of $\cal V$ and a  sum of the $\cal L_i$.

We construct a basis $\{f_1,f_2,e_1,e_2\}$ of $\cal V$ which is orthogonal for $Q$ and such that $Q(f_1,f_1)=Q(f_2,f_2)=1$, so that $Q|_{\langle e_1,e_2\rangle}$ is negative definite. Let $\calW_0=\calW \oplus \langle e_1,e_2\rangle$ (recall that the restriction of $Q$ to $\cal W$ is negative definite). Choose two independent vectors $v$ and $v'$ in $\calW_0$ whose projection on every $\R e_i$ (for $i\geq 1$) is different from 0. Let $\cal V'=\langle \cal V,v,v'\rangle$; the restriction of $Q$ to $\cal V'$ has signature $(2,4)$. Since $\cal V'$ is finite dimensional, we can choose $x,y,z,t\in \Ii_1(\cal V')=\Ii_1(2,4)$ and an isometry $h_0\in\SO^+(2,4)$ such that the 8-tuple $(x,h_0^+,y,g^+,z,h_0^-,t,g^-)$ is maximal and there is no $h_0$-invariant subspace of $\calV'\subset\calH$ that is invariant by $g$. 


Let $\calW'\subset\calW$ be the orthogonal of $\calV'$, and choose a Hilbert basis of $\calW'$ consisting of vectors which have a non-trivial projection on every $\R e_i$. We  choose $h$ that acts as the hyperbolic isometry $h_0$ of $\calV'$, and $h$ acts on $\calW'$ as $g$ does on $\calW$.
The group generated by $g$ and $h$ doesn't preserve any finite dimensional subspace: since every subspace $\calZ\subset \calH$ which is invariant by  $h$ will be either contained in $\calV'$ (and then it is trivial by assumption if it is also invariant by $g$), or contains a vector whose projection on every $\R e_i$ is not trivial. But then it must contain $\calW_0$, if it is $g$-invariant, and therefore cannot be $h$-invariant.
\end{proof}
\begin{remark}\label{rem}
If $g,h$  are constructed as in the proof of Proposition \ref{prop:dense}, for every integer $n$, the pair $g^n,h^n$ satisfies the conclusion of Proposition \ref{prop:dense} as well.
\end{remark}

Given an interval $I_{a,b}$ we denote its closure by $\ov {I_{a,b}}$ for the quotient topology on the projective space $\P\calH$ coming from the Hilbert topology on $\calH$. The following property of intervals is also useful:
\begin{prop}\label{prop:interval_closure}
Assume $(a,b,c,d)\in\Ii_1(2,\infty)^4$ is maximal, then $\overline {I_{b,c}}\subset I_{a,d}$.
\end{prop}
\begin{proof}
As above we can assume without loss of generality that $\ov b=e_1+e_3$, $\ov c=-e_1+e_3$ for a Hilbert basis orthogonal for $Q$, such that $Q(e_1)=Q(e_2)=1$, $Q(e_i)=-1$ for $i\geq 3$. A generic point $t\in \ov {I_{b,c}}$ will then have a representative of the form 
$\ov t=(\cos\theta_t,\sin\theta_t,v_t,w^t_1,\ldots)$ where $w^t\in \langle e_4,\ldots\rangle$, $\|w^t\|^2+v_t^2=1$, $0\leq \theta_t\leq \pi$ and $v_t\geq|\cos\theta_t|$. A similar computation shows that the classes $a,d$ will have representatives $\ov a,\ov b$ of a similar form such that  $\|w^a\|^2+v^2_a=\|w^d\|^2+v^2_d=1$, $-\pi\leq \theta_d<\theta_a\leq 0$ and $v_a>|\cos\theta_a|$, $v_d>|\cos\theta_d|$.

In order to verify that $(a,t,d)$ is maximal it is enough to verify that $Q(\ov a,\ov t)$ and $Q(\ov d,\ov t)$ are negative (the condition with the orientation follows immediately from the analogue property for intervals in the circle): an explicit computation gives
$$Q(\ov a,\ov t)=\cos\theta_a\cos\theta_t-v_av_t +\sin\theta_a\sin\theta_t-\langle w_a,w_t\rangle<0.$$
More precisely, since $v_a>|\cos\theta_a|$ and $v_t\geq|\cos\theta_t|$
\begin{equation}\label{eq1}\cos\theta_a\cos\theta_t-v_av_t\leq0.\end{equation}
Furthermore, since $-\sin\theta_a>\| w_a\|$ and $\sin\theta_t\geq\| w_t\|$,
\begin{equation}\label{eq2}\sin\theta_a\sin\theta_t-\langle w_a,w_t\rangle<\sin\theta_a\sin\theta_t+\| w_a\|\| w_t\|\leq0\end{equation}
Observe that Equations~\eqref{eq1} and \eqref{eq2} cannot be an equality simultaneously because if Equation~\eqref{eq1} is an equality then $\theta_t=\pi/2$ and thus $\sin(\theta_t)=1$.
The verification that $Q(\ov d,\ov t)<0$ is identical, and thus the result follows.
\end{proof}

Combining Propositions  \ref{prop:intervals}, \ref{prop:dense} and \ref{prop:interval_closure}, we obtain
\begin{cor}\label{cor:pingpong}
There exists a maximal 4-tuple $(x,y,z,t)$ in $\Ii_1(2,\infty)$, and a pair of Shilov-hyperbolic elements $A,B\in\OOO_\R^+(2,\infty)$ that plays ping-pong with this tuple, namely such that 


\begin{minipage}{.5\textwidth}
$$\left\{\begin{array}{l}
A \ov{I_{t,z}}\subset I_{x,y}\\
B \ov{I_{x,t}}\subset I_{y,z}\\
A^{-1} \ov{I_{y,x}}\subset I_{z,t}\\
B^{-1} \ov{I_{z,y}}\subset I_{t,x}.
\end{array}\right.$$ 
\end{minipage}
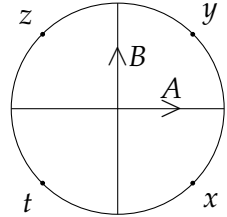
\begin{minipage}{.5\textwidth}
\begin{tikzpicture}[scale=0.6]
\draw (0,0) circle [radius=2];

\draw (-2,0) to (2,0);
\node at (1,0) [above] {$A$};
\node at (1,0)  {$>$};

\draw (0,-2) to (0,2);
\node at (0,1) [right] {$B$};
\node at (0,1)[rotate=90]  {$>$};

\filldraw (1.41,1.41) circle [radius=1pt];
\filldraw (-1.41,1.41) circle [radius=1pt];
\filldraw (-1.41,-1.41) circle [radius=1pt];
\filldraw (1.41,-1.41) circle [radius=1pt];
\node at (1.4,1.4) [above right]{$y$};
\node at (1.4,-1.4) [below right]{$x$};
\node at (-1.4,-1.4) [below left ]{$t$};
\node at (-1.4,1.4) [ above left]{$z$};


\end{tikzpicture}
\end{minipage}
We can furthermore assume that the group generated by $A,B$ doesn't leave invariant any finite dimensional subspace of $\mathcal H$.
\end{cor}
\begin{proof}
Let $g,h$ be the Shilov hyperbolic elements  and $(x,y,z,t)$ be the points  given by Proposition \ref{prop:dense}. 
Proposition \ref{prop:intervals} implies that we can find an integer $n$ such that the pair $(A,B)=(h^n,g^n)$ plays ping pong with the 4-tuple. Morever, we can pass to the closure thanks to Proposition \ref{prop:interval_closure}. The second claim is a consequence of  Remark \ref{rem}.
\end{proof}
\begin{prop}\label{prop:mas}
Let $\Sigma$ be the once punctured torus, and let $a,b$ be the standard generators of $\G_\Sigma=\pi_1(\Sigma)$ oriented as in the picture. 

\noindent
\begin{minipage}{.4\textwidth}
Assume that $\rho:\G_\Sigma\to\OOO_\R^+(2,\infty)$ has the property that the image $\rho(aba^{-1}b^{-1})$ has a fixed point $l$ in $\Ii_1(2,\infty)$. Then 
$$2i_\rho=\beta_\R(l,\rho(a^{-1})l,\rho(ba)^{-1}l)+\beta_\R(\rho(ba)^{-1}l,\rho(b^{-1})l,l).$$
\end{minipage} 
\hspace{3.5cm}
\begin{minipage}{.3\textwidth}
\begin{tikzpicture}[scale=.6]
\draw (0,0) circle [radius=2];

\draw (-2,0) to (2,0);
\node at (1,0) [above] {$a$};
\node at (1,0)  {$>$};

\draw (0,-2) to (0,2);
\node at (0,1) [right] {$b$};
\node at (0,1)[rotate=90]  {$>$};

\draw (1.4,1.4) to (1.4,-1.4); 
\draw (-1.4,1.4) to (-1.4,-1.4); 
\draw (1.4,-1.4) to (-1.4,-1.4); 
\draw (1.4,1.4) to (-1.4,1.4); 
\draw (1.4,1.4) [dashed]to (-1.4,-1.4);
\node at (1.6,1.4) [right]{$l$};
\node at (1.4,-1.4) [right]{$\rho(b)^{-1}l$};
\node at (-1.4,-1.4) [left ]{$\rho(a^{-1}b^{-1})l$};
\node at (-1.4,1.4) [ left]{$\rho(a^{-1})l$};
\end{tikzpicture}

\end{minipage}
\end{prop}

\begin{proof}
As the (relative) bounded cohomology of a surface with a puncture and that of an homotopic surface with a boundary component are canonically isomorphic, we can realize $\Sigma$ as a surface with geodesic boundary $\partial \Sigma$. We denote by $\Hb^2(\Sigma,\R)$ the singular bounded cohomology of the topological space $\Sigma$ (namely the cohomology of the complex of bounded singular cochain), and by  $\Hb^2(\Sigma,\partial\Sigma,\R)$ the relative bounded cohomology, which is the cohomology of the complex of bounded cochains that vanishes on singular simplices with image entirely contained in $\partial \Sigma$.

It follows from \cite[Theorem 3.3]{BIW} that the Toledo invariant $i_\rho$ can be computed from the formula
$$2i_\rho=\langle j_{\partial\Sigma}^{-1}g_\Sigma\rho^*(\kappa_b),[\Sigma,\partial\Sigma]\rangle.$$
Here $g_\Sigma:\Hb^2(\G_\Sigma,\R)\to\Hb^2(\Sigma,\R)$ is the canonical isomorphism, and $j_{\partial\Sigma}^{-1}:\Hb^2(\Sigma,\R)\to\Hb^2(\Sigma,\partial\Sigma,\R)$ is the isometric isomorphism described in \cite{BBIFPP} that is inverse to the map induced by the inclusion of bounded relative cochains in bounded cochains. Recall that, whenever a base point $x\in\widetilde \Sigma$, the universal cover, is fixed, the bounded cohomology $\Hb^2(\Sigma,\R)$ can be also isometrically computed from the complex of functions on straight simplices with vertices in the set $\G_\Sigma\cdot x\subset \widetilde\Sigma$. Furthermore if $c$ is a cocycle representing the class $[c]\in\Hb^2(\Gamma_\Sigma,\R)$, the class $g_\Sigma([c])$ is represented by the cocycle 
$$\ov c(\Delta(g_0x,g_1x,g_2x))=c(g_0,g_1,g_2).$$
We denote, as in Lemma \ref{lem:5.4}, $C_\beta^l\in C^2_b(\Gamma_\Sigma,\R)$ the cocycle defined by 
$$ C_\beta^l(g_0,g_1,g_2)=\beta_\R\left(\rho(g_0)l,\rho(g_1)l,\rho(g_2)l\right)$$ 
(recall that $l\in\Ii_1(2,\infty)$ is a fixed point of $\rho(bab^{-1}a^{-1}))$).
We deduce that $\ov C_\beta^l$ vanishes on simplices contained in $\partial\Sigma$, as long as we choose $x\in\widetilde\Sigma$ in the pre-image of $\partial\Sigma$. Thus 
$$\langle j_{\partial\Sigma}^{-1}g_\Sigma\rho^*(\kappa_b),[\Sigma,\partial\Sigma]\rangle=\sum_i a_i\beta_\R(g^i_0l,g^i_1l,g^i_2l),$$
provided $\sum_i a_i\Delta(g^i_0x,g^i_1x,g^i_2x)$ represents the relative fundamental class $[\Sigma,\partial \Sigma]$.

Observe that a relative fundamental class for the once punctured torus can be written as the sum of the triangles $\Delta (x,a^{-1}x,b^{-1}a^{-1}x)$, $\Delta (x, b^{-1}a^{-1}x,a^{-1}b^{-1}x)$ and $\Delta (a^{-1}b^{-1}x,b^{-1}x,x)$, and the cocycle $\beta_\R$ vanishes on the third simplex since $\beta_\R$ is $\G_\Sigma$-equivariant and alternating, and $\Delta (x, b^{-1}a^{-1}x,a^{-1}b^{-1}x)=\Delta (abx, x,[a,b]x)$. 
The result follows.
\end{proof}

\begin{proof}[Proof of Theorem \ref{thm:dense}]
Let $A,B\in\OOO_\R^+(2,\infty)$ as given by Corollary \ref{cor:pingpong}. The group $\Gamma_\Sigma$ is a free group on two generators $a$ and $b$.
We define the representation $\rho$ by setting $\rho(a)=A$, $\rho(b)=B$. Corollary \ref{cor:pingpong} implies that $\rho(bab^{-1}a^{-1})I_{y,z}\subset I_{y,B^+}$.  

Since the interval $I_{y,z}$ is a non-empty bounded convex set of a Hilbert space (Proposition \ref{prop:Iconv}) whose closure is compact in the weak topology, we deduce using Tychonoff fixed point theorem \cite{MR1513031} (see \cite{MR1009162} for a modern proof) that the continuous function $\rho([a,b]): \ov{I_{y,z}}\to \ov{I_{y,z}}$ has a fixed point. 

Since $l$ belongs to the interval $\ov{I_{y,z}}\subset \ov{I_{y,x}}$ we have that $\rho(a^{-1})l$ belongs to the interval $I_{z,t}$ and $\rho(a^{-1}b^{-1})l=\rho(b^{-1}a^{-1})l$ belongs to the interval $I_{t,x}$. This implies that  $$\beta_\R(l,\rho(a^{-1})l,\rho(ba)^{-1}l)=2;$$ the verification that $\beta_\R(\rho(ba)^{-1}l,\rho(b^{-1})l,l)=2$ is analogous. Together with Proposition \ref{prop:mas} this shows that $i_\rho=2$, namely that the representation $\rho$ is maximal.

We conclude the proof verifying that the representation is geometrically dense. Since the representation is irreducible, there is no fixed point at infinity and thus there is a minimal totally geodesic invariant subspace, which can’t be of finite dimension because of Lemma~\ref{lem:fdtgs}. It has no Euclidean factor otherwise there would be a fixed point at infinity or a pair of such fixed points, which is impossible thanks to Proposition~\ref{prop:finite_config}. So either it is of rank 1, a product of two  rank 1 subspaces or a rank 2 subspace. Lemma \ref{lem:tight} excludes the presence of rank 1 factors and that the symmetric subspace is associated to $\OH(2,\infty)$. The closure cannot be associated to $\OC(2,\infty)$ by Theorem \ref{thm:hyplat}. Therefore the minimal totally geodesic subspace is isometric to the symmetric subspace  to $\OR(2,\infty)$, by possibly restricting to the isometry group of that subspace we can assume that the representation is geometrically dense.
\end{proof}

\appendix
\section{Exotic actions of ${\rm PSL}_2(\R)$ on $\Xx_\R(2,\infty)$}\label{Appendix}
Delzant and Py \cite{MR2881312} initiated a geometric study of representations $\pi_s$ of $\PU(1,1)\simeq\PSL_2(\R)$ on the space $\LL^2(\SS^1,\C)$ of square integrable, complex valued functions on the circle $\SS^1=\partial \mathbf D$, seen as the boundary of the unit disk $\mathbf D$,  endowed with the angular measure ${\rm d}\theta/2\pi$. While these representations were previously studied from an algebraic point of view, they noticed that they give rise to interesting exotic actions on infinite dimensional symmetric spaces of finite rank. Despite the main interest of \cite{MR2881312} (as well as of  \cite{MR3263898}) being actions on the infinite dimensional real hyperbolic space, the  construction also gives a one parameter family of representations in $\OR(2,\infty)$. The goal of this appendix is to explicitly compute the Toledo invariant of those representations. We will show that the invariant vanishes.

We quickly recall the construction in our specific setting. We refer the reader to \cite[Section 2]{MR2881312} for more details. Let $s\in(3/2,5/2)$. The representation $\pi_s$ alluded to before is defined by
$$\pi_s(g)\cdot f=\Jac(g^{-1})^{\frac12+s} f\circ g^{-1},$$
where $\Jac(g)$ is the Jacobian of an element $g$ with respect to the measure ${\rm d}\theta$ on the circle.

If we denote by $c$ the constant function and, for every $n\in\mathbf{Z}\setminus \{0\}$, we denote by $e_n,f_n$ the functions $z\mapsto \Re(z^n), z\mapsto\Im(z^n)$ which are the real and the imaginary part of $z\mapsto z^n$ (these constitute a Hilbert basis of the space $\LL^2(\SS^1,\R)$), then the representation $\pi_s$ is not unitary, but it is shown in \cite[Proposition 2]{MR2881312} that $\pi_s$ preserves a quadratic form $Q_s$ for which the family $\{c,e_i,f_i\}$ is orthogonal and satisfies  
$$Q_s(e_n)=Q_s(f_n)=-\prod_{i=0}^{n-1} \frac{i+\frac 12-s}{i-\frac12+s}$$

and $Q(c)=-1$. It is easy to compute that, for every $s\in(3/2,5/2)$, $Q_s(e_n)<0$ if $n\neq 1$ and $Q_s(e_1)=Q_s(f_1)>0$, and hence the action of $\pi_s$ on the completion $\calH$ of $\LL^2(\SS^1,\R)$ with respect to the form $Q_s$ induces an homomorphism in $\OR(2,\infty)$.  The purpose of the section is to prove the following.
\begin{prop}\label{prop:A1}
Let $\G<\SU(1,1)$ be a torsionfree lattice, and let $\rho_s:\G\to\OR(2,\infty)$ denote the restriction to $\G$ of the  composition of the projection to $\PU(1,1)$ and  $\pi_s$. Then $\rho_s^*\kappa^b_{\OR(2,\infty)}=0$.
\end{prop}

We denote by $\Xx^s_\R(2,\infty)$ the symmetric space associated to the group preserving the form $Q_s$. Since the subgroup $\U(1)<\SU(1,1)$ fixes the positive definite subspace $x=\langle e_1,f_1\rangle\in\Xx^s_\R(2,\infty)$, we have a $\PU(1,1)$-equivariant (harmonic) map $f_s:\mathbf D=\Xx_\C(1,1)\to\Xx^s_\R(2,\infty)$ induced by the orbit map $g\mapsto gx$. Let $\omega_s$ denote the K\"ahler form of the symmetric space $\Xx^s_\R(2,\infty)$; and let us denote by $\Sigma$ the quotient $\mathbf D/\Gamma$. 

We prove the stronger fact that $f_s$ is a totally real equivariant harmonic map, that is 
$\omega_s (df_s( v), df_s (Jv))=0$ for some vector $v\in T_0\mathbf D$ (here $J$ denotes the complex structure of the disk, which we identify as the corresponding element in $\U(1)<\SU(1,1)$).  
  For this purpose, we consider the one parameter subgroup of hyperbolic elements 
  $$g_t=\begin{pmatrix}\cosh t&\sinh t\\\sinh t&\cosh t\end{pmatrix}$$
  whose axis contains 0. Let us denote by $\gamma:\R^+\to \mathbf D$ the geodesic $\gamma(t)=g_t\cdot 0$, and let $v=\gamma’(0)$. In order to compute the image $df_s( v)$, we will compute $\left.\frac{\rm d}{{\rm d} t}\right|_{t=0}\pi_s(g_t)\cdot x$. Observe that $\pi_s(g_t)\cdot x$ is the vector space generated by the real and imaginary part of the function $\pi_s(g_t)\cdot z$ (where, for ease of calculation, we extend the action of $\pi_s$ to the Hilbert space $\LL^2(\SS^1,\C)$).

If we denote by $a=\tanh t$ we have 
$$\pi_s(g_t)\cdot z=\Jac(g_t^{-1})^{\frac12+s}\frac{z-a}{1-az}=(1-a^2)^{\frac12+s}(z-a)\left(\sum_{n=0}^\infty a^nz^n\right)^{2+2s}$$
since
$$\Jac(g_t^{-1})=\frac{1-a^2}{(1-az)^2}.$$
Therefore, we have 
$$\left.\frac{\rm d}{{\rm d} t}\right|_{t=0}\pi_s(g_t)\cdot z=-1+(2+2s)z^2.$$
Using the notation from Section \ref{sec:Hermitian}, we may identify $T_0\calX_\R^s(2,\infty)$ with the Lie triple system $\mathfrak p=\left\{\begin{bmatrix}0&A\\^tA&0\end{bmatrix},\ A\in L(W,V)\right\}$  where $V=\langle e_1,f_1\rangle$ and $W=\langle c,e_2,f_2,\ldots\rangle$. The tangent vector $df_s( v)$ is the element in the tangent space $\mathfrak{p}$  that corresponds to the  matrix $A\in L(W,V)$ given by
$$A=\begin{pmatrix}
-1&2+2s &0 &0 &\ldots\\
0&0&2+2s &0 &\ldots
\end{pmatrix}. $$
Since the vectors $e_{2n}, f_{2n}$ are eigenvectors for $\pi_s(J)$ of eigenvalues $(-1)^n$ we get that the tangent vector  $df_s( Jv)$ corresponds to the matrix 
$$B=\begin{pmatrix}
-1&-2-2s &0&0 &\ldots\\
0& 0&-2-2s &0 &\ldots
\end{pmatrix} .$$
Denoting by $J_0$ the complex structure of $\Xx_\R^s(2,\infty)$, we have that $J_0\cdot df_s( Jv)\in\mathfrak p$ corresponds to the matrix  
$$IB=\begin{pmatrix}
0&0&-2-2s &0 &\ldots\\
1& 2+2s&0 &0 &\ldots
\end{pmatrix}. $$
Since $\left[\begin{array}{cc}0&A\\^t A&0\end{array}\right]$ and $\left[\begin{array}{cc}0&IB\\-^t( IB)&0\end{array}\right]$ are orthogonal with respect to the scalar product on $\SSS^2(\calH)$, we obtain our claim and conclude the proof of Proposition \ref{prop:A1}.
\bibliographystyle{alpha}
\bibliography{../biblio.bib}
\end{document}